\documentclass[a4paper]{article}
\usepackage[utf8]{inputenc}
\usepackage[T1]{fontenc}
\usepackage{indentfirst}
\usepackage[a4paper,top=2.54cm,bottom=2.0cm,left=3.0cm,right=3.54cm]{geometry}

\usepackage[round,comma]{natbib}

\usepackage{hyperref}
\usepackage{float}

\usepackage{amsmath}
\usepackage{amssymb}
\usepackage{amsthm}

\usepackage{tikz}
\usetikzlibrary{arrows.meta}
\tikzset{
    circ/.style={draw, circle,inner sep=0pt,minimum size=8mm, font=\scriptsize},
    triangle/.tip={Computer Modern Rightarrow[open,angle=120:3pt]}
}

\newtheorem{theorem}{Theorem}
\newtheorem{proposition}[theorem]{Proposition}
\newtheorem{remark}{Remark}

\newtheorem{assumption}{Assumption}

\def\P{{\mathbb P}}
\newcommand{\E}{\mathbb{E}}

\newcommand{\R}{\mathbb{R}}

\newcommand{\A}{\mathcal{A}_N}

\newcommand{\st}{\R^N}
\newcommand{\opn}{\mathcal{O}}

\newcommand{\mydef}{:=}

\title{\textbf{\Large{Propagation of chaos and phase transition in a stochastic model for a social network}}}

\author{Eva Löcherbach and Kádmo Laxa}

\date{\today}

\begin{document}

\maketitle


\begin{abstract}
We consider a model for a social network with N interacting social actors. This model is a system of interacting marked point processes in which each point process indicates the successive times in which a social actor expresses a “favorable” (+1) or “contrary” (-1) opinion. The orientation and the rate at which an actor expresses an opinion is influenced by the social pressure exerted on this actor. The social pressure of an actor is reset to 0 when the actor expresses an opinion, and simultaneously the social pressures on all the other actors change by h/N in the direction of the opinion that was just expressed.
We prove propagation of chaos of the system, as N diverges to infinity, to a limit nonlinear jumping stochastic differential equation. Moreover, we prove that under certain conditions the limit system exhibits a phase transition described as follows. If h is smaller or equal than a certain threshold, the limit system has only the null Dirac measure as an invariant probability measure, corresponding to a vanishing social pressure on all actors. However, if h is greater than the threshold, the system has two additional non-trivial invariant probability measures. 
One of these measures has support on the positive real numbers and the other is obtained by symmetrization with respect to $ 0,$ having thus support on the negative real numbers.

\textbf{Keywords: } mean-field interaction, interacting point processes with memory of variable length, propagation of chaos, phase transition, social networks.

AMS MSC: 60K35, 60G55, 91D30. 
\end{abstract}

\section{Introduction}

We study a system of interacting marked point processes with memory of variable length modeling a social network. This system can be informally described as follows.

The system is composed of $N\geq 2$ social actors expressing opinions about a certain subject. Each actor is associated to a marked point process that indicates the successive times in which this actor expresses a favorable ($+1$) or a contrary opinion ($-1$). At each time, each actor is associated to its social pressure which is a real value.
The rate in which an actor expresses a $+1$ opinion is a function of the social pressure of the actor at that time and the rate in which an actor expresses a $-1$ opinion is a function of minus the social pressure of the actor at that time.

When an actor expresses a $+1$ opinion, its social pressure is reset to $0,$ and the social pressures of the other actors increase by $h/N$, where $h>0$. Similarly, when an actor expresses a $-1$ opinion, its social pressure is reset to $0,$ and the social pressures of the other actors decrease by $h/N$.

The parameters of the model are the interaction strength $ h > 0 $ and the jump rate function. Concerning the latter, we work under fairly general assumptions; the precise assumptions imposed both on the jump rate and the initial list of social pressures of the system will be specified in Section \ref{sec:model}. Let us now informally present our results.

Consider the stochastic process describing the time evolution of the social pressure of a given actor in a network with $N$ actors. The sequence of these stochastic processes for $N=2,3,\ldots$, is tight. Moreover, the sequence of the associated empirical measures is tight. This is the content of Theorem \ref{teo:tight}.

Since our model has mean field interactions with $1/N$ scaling, some heuristics allow us to ``guess'' the limit stochastic differential equation that should describe the behavior of the social pressure of an actor when $N\to \infty$. In Theorem \ref{teo:sde} we show that this limit differential equation is well posed. Theorem \ref{teo:convergence} completes this fact by establishing the strong {\it propagation of chaos} property. Here we show that the finite system converges to the limit equation and we provide a rate of convergence for the strong error. In particular, asymptotically, 
the evolution of the social pressures of a finite set of actors is described by independent copies of the limit equation. 

Finally, we study the invariant measures of the limit equation. In particular, under a strong structural assumption on the function describing the jump rates of the system, we have a phase transition: If $h$ is smaller or equal than a given threshold, the limit system has only the null Dirac measure as an invariant probability measure. If $h$ is greater than the threshold, then the system possesses, in addition to the null Dirac measure, two nontrivial invariant probability measures which are explicitly given and which are symmetric from each other. This is the content of Theorem \ref{teo:invariant}.

The model studied in the present article is a generalization of the one introduced in \cite{galveslaxa}. In \cite{galveslaxa}, only exponential rate functions where considered, modulated by a 
\textit{polarization coefficient}, and the paper was devoted to the study of metastability when the polarization coefficient diverges, while keeping the system size fixed.  

On the contrary to this first article, in the present article we work in a mean field frame with interactions of order $1/N,$ allowing for more general jump rate functions, with a focus on the large population behavior. 

Since the social pressure of an actor is reset to $0$ each time it expresses an opinion, this model is a system of interacting point processes with memory of variable length. Indeed, our model belongs to the  class of systems of interacting point process with memory of variable length that was introduced in discrete time by \cite{glmodel} and in continuous time by \cite{gl4} to model systems of spiking neurons. 
A main difference with respect to these articles is that here 
we consider marked point processes to keep track of the opinions ($+1$ and $-1$). 

For a self-contained and neurobiologically motivated presentation of this class of variable length memory models for system of interacting point processes, 
we refer the reader to \cite{gb}. We refer to \cite{galveslaxa} 
for the motivation to consider a social network model on this class. Finally, we refer to \cite{wasserman1,socialdynamics} and \cite{aldous} for a general review on opinion dynamics in a social network.

Mean field limits and propagation of chaos for systems of spiking neurons described by interacting point processes with variable length memory were proven by \cite{gl4, evafour, cormier} and  \cite{evamonm}. Related questions are widely studied in other models of interacting point processes, such as  systems of interacting Hawkes processes (see \cite{dfh, chevallier} and \cite{evadashaerny}). We stress that interacting Hawkes processes usually have memory of infinite length. On the contrary to this, in the present article, we consider models with memory of variable length, induced by the reset of the social pressures. This reset also introduces a strong discontinuity and, as a matter of fact, non Lipschitz terms in the compensator of the jump measure that have to be dealt with carefully. These specific characteristics of our model make our results new and original. 

Our article is organized as follows. In Section \ref{sec:model} we define precisely the model, present the basic notation and state the main results. In Section \ref{sec:priori} we provide a priori bounds for the social pressure of an actor, the limit equation and other related quantities. These upper bounds make use of the reset and are important tools to prove our results.  In particular, we use them to control the non Lipschitz terms induced by the reset jumps; this control is based on a clever truncation procedure in the proof of Theorem \ref{teo:convergence} which is entirely new. Moreover, we prove Theorem \ref{teo:bound}, stated in Subsection \ref{subsec:teobound}, that gives an upper bound for the social pressure of an actor without the dependency on the initial list of social pressures.
In Sections \ref{sec:tight}, \ref{sec:sde}, \ref{sec:convergence} and \ref{sec:invariant} we prove Theorems \ref{teo:tight}, \ref{teo:sde}, \ref{teo:convergence} and \ref{teo:invariant}, respectively.

\section{The model, basic notation and results} \label{sec:model}

Let $\A=\{1,2,...,N\}$ be the set of social actors, with $N \geq 3$, and let $\mathcal{O}=\{-1,+1\}$ be the set of opinions that an actor can express, where $+1$ (respectively, $-1$) represents  a \textit{favorable} opinion (respectively, a \textit{contrary} opinion).

A list of social pressures $u=(u(a): a \in \A)$ is a list in which $u(a)$ is a real number indicating the social pressure of actor $a \in \A$. To describe the time evolution of the social network we introduce a family of maps on the set of lists of social pressures. For any actor $a \in \A$, for any opinion $o \in \mathcal{O}$ and for any list of social pressures $u=(u(a): a \in \A)$, 
we define the new list $\pi^{a,o}_{h,N} (u)$, for a fixed $h \in (0, +\infty)$ as follows. For all $b \in \A$, 
$$
\pi^{a,o}_{h,N}(u)(b)\mydef
\begin{cases}
u(b)+oh/N, \text{ if } b\neq a, \\
0, \text{ if } b  = a.
\end{cases}
$$

The time evolution of the list of social pressures $(U_t^{N})_{t\in [0,+\infty)}$ is a Markov jump process taking values in the set $\st$
with infinitesimal generator defined as follows.
\begin{equation} \label{generator}
\mathcal{G}f(u)\mydef \sum_{o \in \mathcal{O}}\sum_{b\in \A}
\phi(o u(b))\left[f(\pi^{b,o}_{h,N}(u))-f(u)\right],
\end{equation}
for any  bounded function $f:\st \to \mathbb{R}$. Here, $\phi: \mathbb{R}\to [0,+\infty)$ is the jump rate function.

\begin{remark}
Note that for any $N\geq 2$ and for any initial list of social pressures, $(U_t^{N})_{t\in [0,+\infty)}$ will assume values in the set
$$
\left\{u=(u(a): a \in \A) \in \left\{\frac{h}{N}z: z \in \mathbb{Z}\right\}^N:\min\{|u(a)|:a \in \A\}=0\right\}
$$
after the first time in which all actors expressed at least one opinion on the network. For this reason, the state space of the model considered in \cite{galveslaxa} is
$$
\{u=(u(a): a \in \A) \in  \mathbb{Z}^N:\min\{|u(a)|:a \in \A\}=0\}.
$$
However, since we want to consider the limit behavior of the model when $N\to \infty$, we need to consider a state space that is independent of $N$, for each single social actor. 
\end{remark}



In what follows, we consider a realization of the process $(U_t^N)_{t \in [0,\infty)}$ driven by $(\mathbf{N}^{a,o}(ds,dz): a\in \A, o \in \opn)$, which is an i.i.d. family of Poisson random measures on $\R^+ \times \R^+$ having intensity $dsdz$. For any $a\in \A$,
\begin{equation}
\label{eq:finitemodel}
U_t^{N}(a)=U_0^{N}(a)-\sum_{o \in \mathcal{O}}\int_0^t \int_0^{+\infty} U_{s-}^{N}(a)\mathbf{1}\{z \leq \phi(o U_{s-}^{N}(a))\}\mathbf{N}^{a, o}(ds,dz)
\end{equation}
$$
+ \frac{h}{N}\sum_{o \in \mathcal{O}}\sum_{b\neq a}\int_0^t \int_0^{+\infty} 
o\mathbf{1}\{z \leq \phi(oU_{s-}^{N}(b))\}\mathbf{N}^{b,o}(ds,dz).
$$


This representation suggests that, 
in the case in which the initial social pressures of the actors are i.i.d random variables with distribution $g_0$, in the limit as $N\to \infty$, each actor should be the solution of the following differential equation
\begin{equation} \label{eq:limiteq}
U_t=U_0-\sum_{o \in \opn}\int_0^t \int_0^{+\infty} U_{s-}\mathbf{1}\{z \leq \phi(oU_{s-})\}\mathbf{N}^o(ds,dz)
+h\sum_{o\in \mathcal{O}}o \int_0^t \E[\phi(oU_{s})]ds,
\end{equation}
where $U_0\sim g_0$, and $\mathbf{N}^{+1}(ds,dz)$ and $\mathbf{N}^{-1}(ds,dz)$ are independent Poisson random measures  on $\R^+ \times \R^+$ having intensity $dsdz$.


The remainder of this article is devoted to the proof of the strong convergence of the process $(U_t^N(a))_{t\in [0,\infty)}$, represented in \eqref{eq:finitemodel}, to $(U_t)_{t\in [0,\infty)}$, defined in \eqref{eq:limiteq}, and to the study of the long time behavior of the limit process.

To state our main results, we need to introduce some notation. For $r\geq 0$, we put
\begin{equation} \label{eq:bigphi}
\Phi(r)\mydef \phi(-r)+\phi(r).    
\end{equation}
We also denote for $l\geq 0$,
\begin{equation} \label{eq:msmaller}
M^<(l) \mydef \sup\{ \Phi(r): r\in [0,l]\} 
\end{equation}
and
\begin{equation} \label{eq:mgreater}
M^>(l) \mydef \inf\{ \Phi(r): r> l\} 
. 
\end{equation}
Finally, in the following, we will denote for any $\lambda>0$, 
\begin{equation}\label{def:ptpoissonproc}
(P_t(\lambda))_{t\geq 0}
\end{equation}
a homogeneous Poison process with rate $\lambda>0$.

Before introducing our main results, we will present all the assumptions considered in our main results. Note that the assumptions are not used in all the results.

\begin{assumption} \label{as:phi1}
The function $\phi:\R\to \R^+$ is locally bounded. This means that $M^<(r)<\infty$, for any $r\geq 0$.
\end{assumption}

Assumption \ref{as:phi1} states that the function $\phi:\R\to \R^+$, that describes the jump rates of our model, is bounded on the interval $[0,r]$, for any $r\geq 0$. Note that we will not necessarily assume that $\phi:\R\to \R^+$ is continuous. 

\begin{assumption} \label{as:teo1}
$(U_0^{N}(a): a\in \A, N \geq 2)$ are i.i.d. random variables with finite mean.     
\end{assumption}


\begin{assumption} \label{as:phi2}
The function $\phi:\R\to \R^+$ 
is locally Lipschitz.
\end{assumption}

Clearly, if the function $\phi:\R\to \R^+$ satisfies Assumption \ref{as:phi2}, then it also satisfies Assumption \ref{as:phi1}.

\begin{assumption} \label{as:boundedandcont}
There exists $L>0$ such that
$\P(|U_0|\leq L)=1$. 
\end{assumption}



For the next assumptions, recall that  $\Phi: \R^+\to \R^+ $ is defined in \eqref{eq:bigphi}, in terms of the function $ \phi $.

\begin{assumption} \label{as:phi3}
The function $\phi:\R\to \R^+$ satisfies $\Phi(r)>0$, for all $r>0$.
\end{assumption}

\begin{assumption} \label{as:phi4}
For any $\lambda >0$,  $\phi:\R\to \R^+$ satisfies  
$$
\int_0^{\infty}\exp\left(-\lambda\int_0^s \Phi(r)dr\right)ds <\infty.
$$
\end{assumption}

\begin{assumption}\label{as:sym}
For any $x\in \R$, $\phi(x)=f(x)+B$, where $B>0$ and $f(x)$ is an odd function belonging to class $\mathcal{C}^1$, with $f'(x)$ strictly decreasing  in $[0,\infty)$ and satisfying 
$f(0)=0$, 
$f(x)\leq B
$ and $f'(x)>0$, for all $x\in \R$.
\end{assumption}

Note that if the function $\phi:\R\to \R^+$ satisfies Assumption \ref{as:sym} then it also satisfies Assumptions \ref{as:phi1}, \ref{as:phi3} and \ref{as:phi4}.

We can now state our main results. Their proofs are given in Sections \ref{sec:tight}, \ref{sec:sde}, \ref{sec:convergence} and \ref{sec:invariant} below.

Theorem \ref{teo:tight} is about the tightness of $((U_t^N(a))_{t\geq 0}:N\geq 2)$ for $a\in \A$ and the associated empirical measures.  $\mathbb{D}(\R^+)$ denotes the set of càdlàg functions on $\R^+$, which is endowed with the topology of the
Skorokhod convergence on compact time intervals (see \cite{jacshi}). $\mathcal{P}(\mathbb{D}(\R^+))$ denotes the set of probability measures on $\mathbb{D}(\R^+)$ endowed with the topology of weak convergence. 

\begin{theorem} \label{teo:tight} Under Assumptions \ref{as:phi1} and  \ref{as:teo1} the following holds.
\begin{enumerate}
\item For any $a\in \A$, $((U_t^N(a))_{t\geq 0}:N\geq 2)$ is a tight sequence of processes in $\mathbb{D}(\R^+)$.

\item The sequence of empirical measures $\mu_N=N^{-1}\sum_{a\in \A}\delta_{(U_t^N(a))_{t\geq 0}}$ is tight in $\mathcal{P}(\mathbb{D}(\R^+))$.
\end{enumerate}
\end{theorem}

We now turn to the well-posedness of the limit equation \eqref{eq:limiteq}.

\begin{theorem} \label{teo:sde}
Under Assumptions \ref{as:phi2} and \ref{as:boundedandcont} the following holds.
\begin{enumerate}
\item Path-wise uniqueness holds for the nonlinear SDE \eqref{eq:limiteq} in the class of processes $(U_t)_{t\geq 0}$ satisfying $|U_0|<L$, for any fixed $L>0$.
    
\item There exists a unique strong solution $(U_t)_{t\geq 0}$ of the SDE \eqref{eq:limiteq}  satisfying $|U_0|<L$, for any fixed $L>0$.
\end{enumerate}
\end{theorem}

Theorem \ref{teo:convergence} is about a quantified version of the convergence of the model \eqref{eq:finitemodel} to the limit equation \eqref{eq:limiteq}. It implies the propagation of chaos property.

\begin{theorem} \label{teo:convergence}
For $N\geq 2$, let $((U_t(a))_{t\geq 0}: a\in \A)$ be $N$ independent copies of \eqref{eq:limiteq} driven by $(\mathbf{N}^{a,o}: a \in \A, o \in \opn)$ with $U_0(a)=U_0^N(a)$, for $a \in \A$. Suppose that Assumptions  \ref{as:teo1}, \ref{as:phi2} and \ref{as:boundedandcont} hold. 
Then, for any $T>0$ there exists $C=C(L,T,h,\phi)>0$, such that for all $ t \le T $ and for any $a\in \A$,
$$
\E\left(\sup_{s\leq t}|U_s^N(a)-U_s(a)|\right)\leq CN^{-1/2}.
$$
In the above statement, the constant $C=C(L,T,h,\phi)$ is an increasing function of the support of the initial condition, $L,$ the interaction strength $h$ and the time horizon $T$, and it also depends on the function $\phi$.
\end{theorem}

Finally, Theorem \ref{teo:invariant} deals with the long-time behavior of the limit process. We first study the invariant states of the limit equation \eqref{eq:limiteq} under general conditions on the jump rate. For a jump rate satisfying the more restrictive Assumption \ref{as:sym}, we then prove that the limit equation exhibits a phase transition.

\begin{theorem} \label{teo:invariant}
\begin{enumerate}
\item[]
\item Under Assumption \ref{as:phi3}, $\delta_0$ is an invariant probability measure of \eqref{eq:limiteq}.

\item Under Assumptions \ref{as:phi3} and \ref{as:phi4}, any non-null invariant probability measure of \eqref{eq:limiteq} has either support in $\R^+$ or support in $\R^-$.

\item Under Assumptions \ref{as:phi3} and \ref{as:phi4},
any non-null invariant probability measure of \eqref{eq:limiteq} with support in $\R^+$
has the form
$g(dx)=g(x)dx$, with $g: \R^+\to \R^+$ given by
$$ 
g(x)=
g_{\gamma}(x)=\exp\left(-\frac{1}{\gamma h}\int_0^{x} \Phi(s)ds\right) \Big/ \int_0^{\infty}\exp\left(-\frac{1}{\gamma h}\int_0^t \Phi(s)ds\right)dt,
$$
where $\gamma>0$ satisfies $\int_0^{\infty}(\phi(x)-\phi(-x))g_{\gamma}(dx)=\gamma>0$ and $\int_0^{\infty}g_{\gamma}(dx)=1$. 

\item Under Assumptions \ref{as:phi3} and \ref{as:phi4}, $g(dx)=g(x)(dx)$ is an invariant probability measure of \eqref{eq:limiteq} with support in $\R^+$ if and only if $\tilde{g}(dx)=\tilde{g}(x)(dx)$, given by $\tilde{g}(x)=g(-x)$, is an invariant probability measure of \eqref{eq:limiteq} with support in $\R^-$.

\item If Assumption \ref{as:sym} holds, then: 

If $h \leq  B/f'(0)$, then the SDE \eqref{eq:limiteq} has $\delta_0$ as unique invariant probability measure. 

If $h> B/f'(0)$, the SDE \eqref{eq:limiteq} has three invariant probability measures: $\delta_0$, an invariant probability measure supported in $\R^+$ and other supported in $\R^-$. 
\end{enumerate}
\end{theorem}


\begin{remark}
   Note that $\delta_0$ is an invariant probability measure for \eqref{eq:limiteq} because in the case
$$
\E(\phi(U_s))=\E(\phi(-U_s)),
$$
for all $s\geq 0$, corresponding to the case $\gamma=0$, \eqref{eq:limiteq} is a process that remains at $U_0$ until the first jump of the process, which is an exponential random time, and after the first jump time the process gets trapped at $0$. 
\end{remark}

\begin{remark}
From a modeling point of view, it is reasonable to assume that $\phi:\R\to \R^+$ is non decreasing: the greater the social pressure in a certain direction, the greater the tendency of the social actor to express an opinion corresponding to that direction. However, our results are more general. In the case in which $\phi:\R\to \R^+$ is non decreasing, Assumption \ref{as:phi1} always holds.
\end{remark}

Let us comment on the proof of our main results.
We follow the \textit{coupling} approach 
to prove propagation of chaos. This approach goes back to the seminal work of \cite{sznit} and has been successfully implemented to study mean field limits for systems of spiking neurons (see \cite{evafour} and \cite{evamonm}). We need to consider different techniques to deal with the marks on the Poisson processes. Moreover, we need to introduce an adhoc truncation procedure to deal with the fact that we do not assume that $\phi$ is globally Lipschitz to prove the quantified propagation of chaos. 

The proof of part 2 of Theorem \ref{teo:sde} uses a classical Picard iteration argument very much in the spirit of the schemes introduced in \cite{evadasha} and \cite{evadashaerny}. To deal with the case in which $\phi$ is not globally Lipschitz, we first prove the result in the globally Lipschitz case. In a second step, we then put this together with the a priori bounds of Section \ref{sec:priori} to conclude the proof in the local Lipschitz case.

The proof of Theorem \ref{teo:convergence} uses an important truncation trick. We divide the terms we want to control into two cases corresponding to whether the absolute value of the social pressure of an actor is smaller than an arbitrary threshold $l>0$ or not.  For the first case, we can easily obtain a control since the social pressures are bounded, relying on the bounds for the limit process contained in Section \ref{sec:priori}. To deal with the values exceeding $l, $ we have to carefully choose $l$ and to rely on the a priori bounds established in Section \ref{sec:priori}. 

The a priori bounds presented in Section \ref{sec:priori} to bound the social pressure of an actor and other related quantities are proven in an original way by bounding the total number of jumps of the system. Similar bounds for system of spiking neurons on the same class are presented in \cite{gl4,ostduarte1} and \cite{evafour}.

\subsection{Phase transition}

Part 5 of Theorem \ref{teo:invariant} states that under Assumption \ref{as:sym}, the limit equation has two different behaviors according to the value of $h$, having only the null Dirac measure as invariant measure if $h$ is smaller than a given threshold, and having, in addition to the null Dirac measure, two nontrivial invariant probability
measures which are given by Part 3 of Theorem \ref{teo:invariant} and which are symmetric from each other.

In particular, considering $\phi(r)=1+tanh(r)$, where $tanh$ is the hyperbolic tangent function, this function $\phi:\R\to \R^+$ clearly satisfies Assumption \ref{as:sym} with $B=1$ and $f'(0)=1$. Therefore, the limit SDE \eqref{eq:limiteq} has a phase transition as described in Theorem \ref{teo:invariant} with threshold equal to $1$. 

As a consequence, the time evolution of the system \eqref{eq:finitemodel} for a sufficiently large number of social actors exhibits different behaviors depending on the value of $h$. If $h<1$, starting from an initial list in which the social pressure of all actors has the same sign, the mean social pressure of the system quickly approaches $0$, remaining close to $0$. If $h>1$, starting from the null list social pressure, either the mean social pressure of the system quickly approaches a given positive value or it quickly approaches a negative value having the same height (in absolute value). Figure \ref{fig:sim1} presents simulations displaying this behavior. 

\begin{figure}[h] 
\centering
\includegraphics[width=7cm]{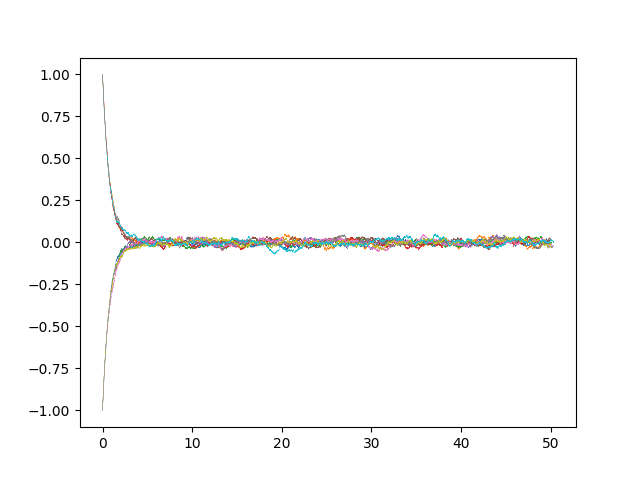} \hfill
\includegraphics[width=7cm]{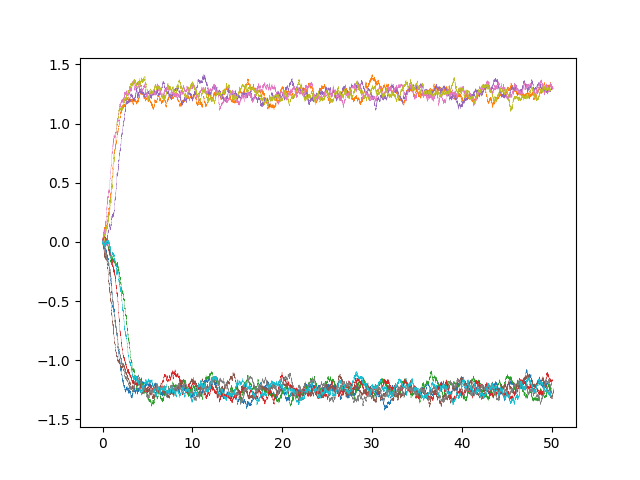}
\caption{Simulation of the system \eqref{eq:finitemodel} with $N=1000$ and $\phi(r)=1+tanh(r)$. The system \eqref{eq:finitemodel} was simulated $10$ times with  $h=0.5$ and $10$ times with $h=2$. When $h=0.5$, we chose the initial list of social pressures $U_0^N \equiv 1$ on half of the simulations and $U_0^N \equiv -1$ on the other half. When $h=2$, $U_0^N \equiv 0$. The time evolution of the mean social pressure of each simulation are the different plots of the figure (in the left for $h=0.5$ and in the right for $h=2$).}
\label{fig:sim1}
\end{figure}

In the case in which $h$ is exactly the threshold ($h=1$), the behavior of the system is  different from the behaviors displayed in Figure \ref{fig:sim1}. Starting from an initial list in which the
social pressure of all actors has the same sign, the mean social pressure of the system quickly
approaches a neighborhood of $0$, but differently from the case $h<1$, the mean social pressure of the system fluctuates around $0$. Figure \ref{fig:hthreshold} presents simulations displaying this behavior.

\begin{figure}[h] 
\centering
\includegraphics[width=7cm]{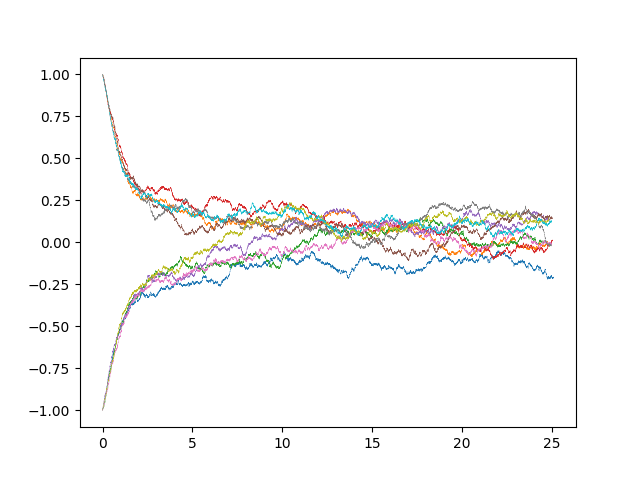}
\caption{Simulation of the system \eqref{eq:finitemodel} with $N=1000$, $\phi(r)=1+tanh(r)$ and $h=1$. The system \eqref{eq:finitemodel} was simulated $10$ times and the initial list of social pressures $U_0^N \equiv 1$ on half of the simulations and $U_0^N \equiv -1$ on the other half.  The time evolution of the mean social pressure of each simulation are the different plots of the figure.}
\label{fig:hthreshold}
\end{figure}

Assumption \ref{as:sym} is clearly not necessary for Part 3 of Theorem \ref{teo:invariant} to hold. Indeed, by considering  $\phi(r)=e^r$, a function that does not satisfy Assumption \ref{as:sym}, the time evolution of the system \eqref{eq:finitemodel} for a sufficiently large number of social actors exhibits the same different behaviors, with the same threshold for $h$. Figure \ref{fig:sim2} presents simulations displaying this behavior.
Even though we can not prove it, we therefore believe that Part 3 of Theorem \ref{teo:invariant} holds for a wider class of functions $\phi:\R\to \R^+$.

\begin{figure}[H] 
\centering
\includegraphics[width=7cm]{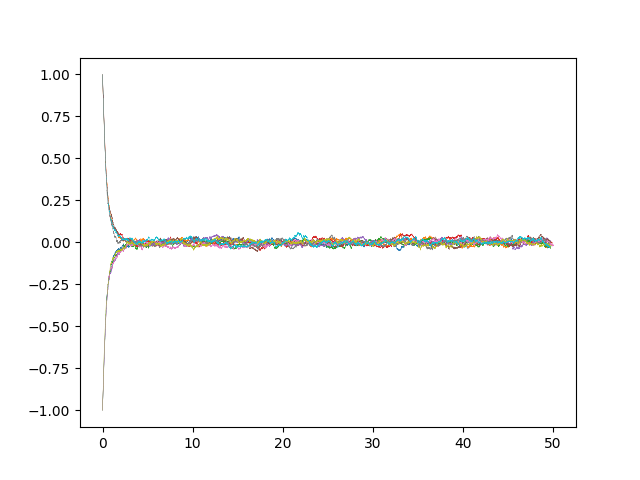} \hfill
\includegraphics[width=7cm]{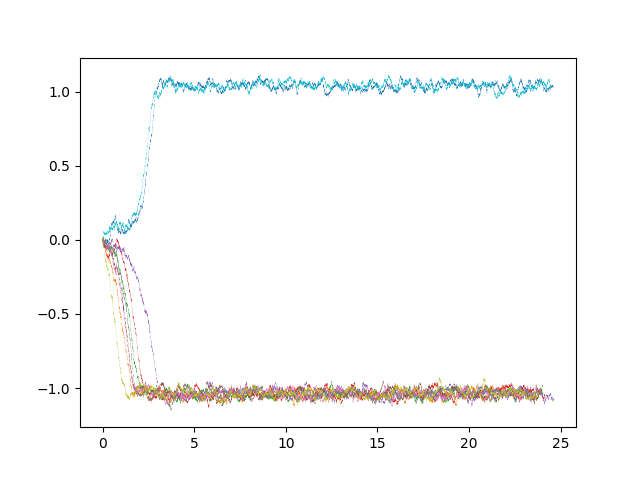}
\caption{Simulation of the system \eqref{eq:finitemodel} with $N=1000$ and $\phi(r)=e^r$. The system \eqref{eq:finitemodel} was simulated $10$ times with  $h=0.5$ and $10$ times with $h=2$. When $h=0.5$, we chose the initial list of social pressures $U_0^N \equiv 1$ on half of the simulations and $U_0^N \equiv -1$ on the other half. When $h=2$, $U_0^N \equiv 0$. The time evolution of the mean social pressure of each simulation are the different plots of the figure (in the left for $h=0.5$ and in the right for $h=2$).}
\label{fig:sim2}
\end{figure}

\section{A priori bounds} \label{sec:priori}
In this Section we provide upper bounds  for $|U_t^N (a) |$ defined in \eqref{eq:finitemodel}, $|U_t|$ defined in \eqref{eq:limiteq} and other related quantities. These upper bounds will be an important tool to prove our results. 
In particular, Theorem \ref{teo:bound}, stated in Subsection \ref{subsec:teobound}, gives an upper bound for the social pressure of an actor not depending on the initial list of social pressures.

\subsection{Bounding the number of jumps of the system}

For any $N\geq 2$ and for any $t\geq 0$, let
\begin{equation}\label{eq:ztndefinition}
Z_t^N\mydef
\sum_{a\in \A}\sum_{o\in \opn}\int_0^t \int_0^{+\infty}\mathbf{1}\{z \leq \phi(o U_{s-}^{N}(a))\}\mathbf{N}^{a, o}(ds,dz)
\end{equation}
be the total number of jumps of the system with $N$ actors until time $t$. 

To upper-bound  $|U^N_t(a)|$, note that from \eqref{eq:finitemodel} it follows that
$$
|U_t^{N}(a)|\leq |U_0^{N}(a)|-\sum_{o \in \mathcal{O}}\int_0^t \int_0^{+\infty} |U_{s-}^{N}(a)|\mathbf{1}\{z \leq \phi(o U_{s-}^{N}(a))\}\mathbf{N}^{a, o}(ds,dz)
$$
$$
+ \frac{h}{N}\sum_{o \in \mathcal{O}}\sum_{b\neq a}\int_0^t \int_0^{+\infty} 
\mathbf{1}\{z \leq \phi(oU_{s-}^{N}(b))\}\mathbf{N}^{b,o}(ds,dz).
$$
Moreover, note that
$$
\sum_{o \in \mathcal{O}}\sum_{b\neq a}\int_0^t \int_0^{+\infty} 
\mathbf{1}\{z \leq \phi(oU_{s-}^{N}(b))\}\mathbf{N}^{b,o}(ds,dz) \leq Z_t^N.
$$
Since $ h > 0, $ this implies that
\begin{equation}
\label{eq:boundsocialandreset}
|U_t^N(a)|+ \sum_{o \in \opn}\int_0^t \int_0^{+\infty} |U_{s-}^{N}(a)|\mathbf{1}\{z \leq \phi(o U_{s-}^{N}(a))\}\mathbf{N}^{a, o}(ds,dz) \leq |U_0^N(a)|+\frac{h}{N}Z_t^N.
\end{equation}
Therefore, in order to upper bound $|U_t^N(a)|$ and the other term on the left-hand side of \eqref{eq:boundsocialandreset}, it is sufficient to find an upper bound for $Z_t^N$. This is done in the next proposition. Recall that $(P_t(\lambda))_{t\geq 0}$ is defined in \eqref{def:ptpoissonproc} for any $\lambda>0$. In the following, for any random variables $X$ and $Y$ assuming values in $\R^+$, we will say that $X \leq Y$ for the usual stochastic order if 
$$
\P(X>r) \leq \P(Y>r),  \text{ for any } r\in \R^+.
$$

\begin{proposition} \label{bound1}
Assume that $M^<(2h)<\infty$. Let $N\geq 2$ and $a\in \A$. For any $t\geq 0$, the following inequalities hold for the usual stochastic order.
\begin{enumerate}
    \item 
$Z_t^{N} \leq N+2\times P_t( N\times M^<(2h))$. 
\item
$\displaystyle 
     \sup_{s\leq t}|U_t^{N}(a)|\leq |U_0^N(a)|+h +\frac{2h}{N}\times P_t(N\times M^<(2h)).$
\item 
$\displaystyle 
\sum_{o \in \opn}\int_0^t \int_0^{+\infty} |U_{s-}^{N}(a)|\mathbf{1}\{z \leq \phi(o U_{s-}^{N}(a))\}\mathbf{N}^{a, o}(ds,dz) \leq$
$$ \quad \quad \quad \quad 
|U_0^N(a)|+h+\frac{2h}{N}\times P_t(N\times M^<(2h)).
$$
\end{enumerate}
\end{proposition}



To prove Proposition \ref{bound1}, we introduce the following notation.
For any $N\geq 2$, $t\geq 0$, $a\in \A$ and $o\in \opn$, let
$$
Z_t^{>,N}(a,o)\mydef
\int_0^t \int_0^{+\infty}\mathbf{1} \{|U_{s-}^{N}(a)|>2h\}\mathbf{1}\{z \leq \phi(o U_{s-}^{N}(a))\}\mathbf{N}^{a, o}(ds,dz),
$$
$$
Z_t^{<,N}(a,o)\mydef
\int_0^t \int_0^{+\infty}\mathbf{1} \{|U_{s-}^{N}(a)|\leq 2h\}\mathbf{1}\{z \leq \phi(o U_{s-}^{N}(a))\}\mathbf{N}^{a, o}(ds,dz),
$$
and
$$
Z_t^N(a,o)\mydef Z_t^{>,N}(a,o)+Z_t^{<,N}(a,o).
$$
Note that
$$
Z_t^N = \sum_{a\in \A}
\sum_{o \in \opn} Z_t^N(a,o).
$$
Analogously, we define $Z_t^{<,N}$, $Z_t^{>,N}$, $Z_t^{<,N}(a)$, $Z_t^{>,N}(a)$ and $Z_t^N(a)$.

For any $t>0$, the equality
$Z_t^N=Z_t^{>,N}+Z_t^{<,N}$ means that the opinions expressed on the system until time $t$ can be divided in two groups according to the value of the social pressure of the social actor that expressed the opinion just before the time in which the opinion was expressed. Either the absolute value of this social pressure is
smaller or equal than the threshold $2h$ (represented by $Z_t^{<,N}$ and named as jumps below $2h$)
or the absolute value of this social pressure is greater than the threshold $2h$ (represented by $Z_t^{>,N}$ and named as jumps above $2h$).

Now we will prove Proposition \ref{bound1}.

\begin{proof}
To prove Part 1 of Proposition \ref{bound1}, first note that the inequality 
$Z_t^{<,N} \leq  P_t( N\times M^<(2h))$ holds for the usual stochastic order by the basic properties of homogeneous Poisson processes. 
Therefore, to prove Part 1 of Proposition \ref{bound1} we need to prove that the following inequality holds 
\begin{equation}
Z_t^{>,N} \leq N+Z_t^{<,N}. \label{eq:bigjumpsbound}
\end{equation}

To prove \eqref{eq:bigjumpsbound}, we will first show that for any $t>0,$
$$
\{Z_t^{>,N} = N+1\} \subset \{Z_t^{<,N} \geq  N+1\} .
$$
Note that
$\{Z_t^{>,N} = N+1\}$ implies that there exists $a\in \A$ such that $Z_t^{>}(a)\geq 2$. 
Let $0< s_1 < s_2 < t$ be times of two jumps above $2h$ of actor $a$. We have that $U_{s_1}^{N}(a)=0$ and $|U_{s_2-}^{N}(a)|\geq 2h$, which implies that
$$
\sum_{b\neq a}\left(Z_{s_2-}^N(b)-Z_{s_1}^N(b)\right) \geq 2N.
$$
Therefore, until time $t$ the actor $a$ expressed an opinion at least twice and the number of times the other actors expressed opinions is at least $2N$. 
This implies that
$$
Z_t^N \geq 2N+2
$$
and then,
$$
Z_t^{<,N} \geq (2N+2)-(N+1)=N+1.
$$

In general, for $n\geq 1$ and for any $t>0$, $\{Z_t^{>,N} = n(N+1)\}$ implies that there exist $0=t_0< t_1 <\ldots <t_n=t$ such that for any $k=1,\ldots, n$, 
$$
Z^{>,N}_{t_k} - Z^{>,N}_{t_{k-1}} = N+1.
$$
This implies that there exists a sequence of actors $a_1,\ldots, a_n$ such that for any $k=1,\ldots, n$,
$$
Z_{t_k}^{>,N}(a_k)-Z_{t_{k-1}}^{>,N}(a_k)\geq 2.
$$
Using the same arguments as above for the case $n=1$, we conclude that 
for any $k=1,\ldots, n$,
$$
Z_{t_k}^N-Z_{t_{k-1}}^N\geq 2N+2
$$
and 
$$
Z_{t_k}^{<,N}-Z_{t_{k-1}}^{<,N}\geq N+1.
$$
Finally, this implies that
$$
Z_t^{<,N}=\sum_{k=1}^n\left(Z_{t_k}^{<,N}-Z_{t_{k-1}}^{<,N}\right) \geq n(N+1).
$$
We conclude that, 
for any $n=1,2,...$ and for any $t>0$
$$
\{Z_t^{>,N} = n(N+1)\} \subset \{Z_t^{<,N} \geq  n(N+1)\}.
$$
Therefore, \eqref{eq:bigjumpsbound} holds, finishing the proof of part 1 of Proposition \ref{bound1}.
Parts 2 and 3 of Proposition \ref{bound1} follow directly by Part 1 of Proposition \ref{bound1} and \eqref{eq:boundsocialandreset}.

\end{proof}


\begin{remark}\label{remarkztbounde}
By following the same steps of the proof of Part 1 of Proposition \eqref{bound1}, we have that,
for the usual stochastic order,
$$
Z_{t}^N-Z_{s}^N \leq N+2\times P_{t-s}(N\times M^<(2h)).
$$
 
\end{remark}

\begin{remark}
\label{coro1} 
As a consequence of  Part 3 of Proposition \ref{bound1}, for any $t>0$,
$$
\int_0^t \E[|U_{s}^{N}(a)| \Phi(U_{s}^{N}(a))]ds \leq \E|U_0^N|+h+2ht M^<(2h).
$$
\end{remark} 

\begin{remark} \label{rem:ztabound} Note that, under Assumption \ref{as:teo1}, the exchangeability of the system implies that
$$
\E(Z_t^N)=\sum_{b\in \A}\E(Z_t^N(b))=N\E(Z_t^N(a)).
$$
Therefore, as a consequence of Part $1$ of Proposition \ref{bound1}, we have that
$$
\E(Z_t^N(a))\leq 1+2tM^<(2h).
$$

\end{remark}

The following two propositions present a bound for $|U_t|$ defined in \eqref{eq:limiteq} and use a trick already present in \cite{evafour}.

\begin{proposition} \label{boundlimit1}
Assume that $M^<(2h)<\infty$ and suppose that $(U_t)_{t\geq 0}$ solves \eqref{eq:limiteq} with initial condition $U_0$ satisfying that $ \E | U_0 | < \infty$. Then, for any $t>0$, the following inequality holds
$$
|U_t| \leq |U_0|+\E|U_0|+2htM^<(2h).
$$
\end{proposition}
\begin{proof}
Recall \eqref{eq:limiteq}:
$$
U_t=U_0-\sum_{o \in \mathcal{O}}\int_0^t \int_0^{+\infty} U_{s-}\mathbf{1}\{z \leq \phi(o U_{s-})\}\mathbf{N}^o(ds,dz)
+h\sum_{o\in \mathcal{O}}o \int_0^t \E[\phi(oU_{s})]ds.
$$
It follows that
$$
|U_t|\leq |U_0|-\sum_{o \in \mathcal{O}}\int_0^t \int_0^{+\infty} |U_{s-}|\mathbf{1}\{z \leq \phi(o U_{s-})\}\mathbf{N}^o(ds,dz)
+h \int_0^t \E[\Phi(U_{s})]ds
$$
and therefore,
\begin{equation} \label{eq:limitbound1}
\E|U_t|-\E|U_0|\leq-\int_0^t  \E[|U_{s}|\Phi(U_s)]ds
+h \int_0^t \E[\Phi(U_{s})]ds=\int_0^t \E[(h-|U_{s}|)\Phi(U_s)]ds.
\end{equation}
Note that
\begin{equation} \label{eq:limitbound2}
\int_0^t \E[(h-|U_{s}|)\Phi(U_s)\mathbf{1}\{|U_s|\leq 2h\}]ds+\int_0^t \E[(h-|U_{s}|)\Phi(U_s)\mathbf{1}\{|U_s|> 2h\}]ds 
\end{equation}
$$
\leq htM^<(2h) - h\int_0^t \E[\Phi(U_s)\mathbf{1}\{|U_{s}|>2h\}]ds.
$$
By putting together \eqref{eq:limitbound1} and \eqref{eq:limitbound2}, we conclude that
$$
h\int_0^t \E[\Phi(U_s)\mathbf{1}\{|U_{s}|>2h\}]ds \leq \E|U_0|+ htM^<(2h).
$$
Since
$$
|U_t|\leq |U_0|
+h\int_0^t \E[\Phi(U_{s})]ds,
$$
this finishes the proof.
\end{proof}

\begin{proposition}\label{boundlimit2}
Assume that $M^<(2h)<\infty$ and suppose that $(U_t)_{t\geq 0}$ solves \eqref{eq:limiteq} with initial condition $U_0$ satisfying that $ \E | U_0 | < \infty$. Then, for any $t>0$, the following inequality holds
$$
\int_0^t \E[|U_{s}|\Phi(U_{s})]ds\leq 2\E|U_0|+2htM^<(2h).
$$
\end{proposition}
\begin{proof}
Recalling \eqref{eq:limitbound1}, we have that
\begin{multline}\label{eq:limitbound11}
\E|U_t|-\E|U_0|\\
\leq
\int_0^t \E[(h-|U_{s}|)\Phi(U_{s})]ds = \int_0^t \E\left[\left(h-\frac{1}{2}|U_{s}|\right)\Phi(U_{s})\right]ds - \frac{1}{2}\int_0^t \E[|U_{s}|\Phi(U_{s})]ds.
\end{multline}
Note that 
$$
\int_0^t \E\left[\left(h-\frac{1}{2}|U_{s}|\right)\Phi(U_{s})\right]ds \leq \int_0^t \E\left[\left(h-\frac{1}{2}|U_{s}|\right)\Phi(U_{s})\mathbf{1}\{|U_s|\leq 2h\}\right]ds \leq htM^<(2h).
$$
Therefore, coming back to \eqref{eq:limitbound11},
$$
0\leq \E|U_0|+ htM^<(2h)-\frac{1}{2}\int_0^t \E[|U_{s}|\Phi(U_{s})]ds.
$$
We finish the proof by rearranging the terms of the equality above.
\end{proof}

\begin{remark}\label{remarkphitilde}
Note that Part 2 of Proposition \ref{bound1} gives us a bound for $\sup_{s\leq t}|U_s^N(a)|$ that depends only on the behavior of the function $\phi$ restricted to the interval $[-2h,2h]$. This implies that this is a general upper bound for the class of functions $\tilde{\phi}$ satisfying
$$
\sup\{\tilde{\phi}(x)+\tilde{\phi}(-x):0\leq x\leq 2h\}=\sup\{\phi(x)+\phi(-x):0\leq x\leq 2h\}.
$$
Note that this includes any function $\tilde{\phi}$ satisfying $\tilde{\phi}(x)=\phi(x)$, for any $x\in [-2h,2h] $.

Similarly, Proposition \ref{boundlimit1} gives us a general upper bound for the class of functions $\tilde{\phi}$ satisfying  the condition above.
\end{remark}

\begin{remark}\label{remark:ctdef}
Under Assumption \ref{as:boundedandcont} and by considering
\begin{equation} \label{def:ct}
\gamma_t:=2L+2htM^<(2h),
\end{equation}
Propositions \ref{boundlimit1} and \ref{boundlimit2} imply that
$$
\sup_{s\leq t} |U_s| \leq \gamma_t
$$
and 
$$
\int_0^t \E[|U_{s}|\Phi(U_{s})]ds\leq \gamma_t.
$$
\end{remark}

\subsection{Upper bounds for the social pressure of an actor
} \label{subsec:teobound}

The aim of this subsection is to provide
bounds that use the reset factor to improve the bounds obtained in Proposition \ref{bound1}. 
Our goal is to obtain a control of the total number of jumps of the process without taking into account the sign of the expressed opinions. Therefore
we will consider another representation of the system, driven by a different family of Poisson random measures. This representation is given as follows. For any $N\geq 2$ and $a\in \A$, let
\begin{equation}
\label{eq:finitemodel2} 
U_t^{N}(a)=U_0^{N}(a)-\int_0^t \int_0^{+\infty} U_{s-}^{N}(a)\mathbf{1}\{z \leq \Phi(U_{s-}^{N}(a))\}\mathbf{N}^{a}(ds,dz)
\end{equation}
$$
+ \frac{h}{N}\sum_{b\neq a}\int_0^t \int_0^{+\infty} 
\left( \mathbf{1}\{z \leq \phi(U_{s-}^{N}(b))\}-\mathbf{1}\{\phi(U_{s-}^{N}(b))<z \leq \Phi(U_{s-}^{N}(b))\}\right) \mathbf{N}^{b}(ds,dz),
$$
where $(\mathbf{N}^{a}(ds,dz): a\in \A)$ is an i.i.d. family of Poisson random measures  on $\R^+ \times \R^+$ having intensity $dsdz$.

We may now state the main theorem of this section.

\begin{theorem}\label{teo:bound}
Suppose that $M^<(2h)<\infty$ and suppose that $\displaystyle\lim_{l\to \infty}M^>(l)=+\infty$. For any $l\geq 0$, let 
$$
(M^>)^{-1}(l)=\inf\{r\geq 0: M^>(r) \geq l\} 
$$ 
be the generalized inverse function of $M^>$.
Let $a\in \A$, for any $0 \le s<t$, let 
$$
E=\inf\{r>0: \mathbf{N}^a([s,t]\times [0,r]) \geq 1\}
$$
and let $t-Z \in [s, t ]$ be the time in which the mark $E$ occurred.
Then there exists a Poisson process $(P_s(N\times M^<(2h)))_{s\geq 0}$ with rate $N\times M^<(2h)$, independent of $\mathbf{N}^a$, such that 
the following inequalities hold (the first inequality holds for the usual
stochastic order). 
\begin{enumerate}
    \item 
$
\displaystyle |U_t^N(a)|\leq (M^>)^{-1}(E) +h+\frac{2h}{N}P_{Z}(N\times M^<(2h)). 
$

\item 
$\displaystyle \E|U_t^N(a)|\leq \inf_{s\in(0,t]}\left(\E((M^>)^{-1}(E))+h+h(t-s)M^<(2h)\right).$
\end{enumerate}
\end{theorem} 

To prove the above theorem, we will need the following notation.
For any $N\geq 2$, let
$$
\tilde{Z}_t^N(a):=\int_0^t \int_0^{+\infty}\mathbf{1}\{z \leq \Phi(U_{s-}^{N}(a))\}\mathbf{N}^{a}(ds,dz)
$$
be the number of jumps of actor $a\in \A$ until time $t$ and
$$
\tilde{Z}_t^N\mydef
\sum_{a\in \A} \tilde{Z}_t^N(a)
$$
be the total number of jumps of the system with $N$ actors until time $t$. 
Note that $\tilde{Z}_t^N(a)$ and $\tilde{Z}_t^N$ are the adapted definitions of $Z_t^N(a)$ and $Z_t^N$ taking into account the representation \eqref{eq:finitemodel2}. 
The use of this representation and this notation is restricted to this subsection.

\begin{remark} \label{remark:zttilde} Recall definition \eqref{def:ptpoissonproc}. For any $t>0$, it follows from Proposition \ref{bound1} that  for the usual stochastic order,
$$
\tilde{Z}_{t}^N \leq N+2\times P_t(N\times M^<(2h)).
$$
Moreover, for any $s<t$, it follows from Remark \ref{remarkztbounde} that for the usual stochastic order,
$$
\tilde{Z}_{t}^N-\tilde{Z}_{s}^N \leq N+2\times P_{t-s}(N\times M^<(2h)).
$$
\end{remark}

To obtain the bounds of this subsection, for a fixed $N\geq 2$ and a fixed $a\in \A$, consider the homogeneous Poisson process with rate $\lambda>0$ given by
$$
(\mathbf{N}^a([0,t]\times [0,\lambda]): t\geq 0).
$$
For any $\lambda>0$, let $(T_n^{\lambda})_{n\geq 1}$ be the jump times of this Poisson process.

To prove Theorem \ref{teo:bound} we will use the following Proposition.
 
\begin{proposition}\label{prop:bigjumpbound}
For $l>0$, suppose that $M^>(l)>0$. Then, for any $n\geq 1$,
$$
|U_{T_n^{M^>(l)}}^N(a)| \leq l.
$$
\end{proposition}
\begin{proof}
Note that, for any $\lambda>0$ and for any $r\in \R$ such that $\Phi(r) \geq \lambda$, it follows that
$$
 \{T_n^{\lambda} :n\geq 1\} \subset \{T_n^{\Phi(r)} :n\geq 1\} .
$$
This implies that for any $n\geq 1$,  $T_n^{\lambda}$ is a jump time of actor $a$ if $\Phi(U_{T_n^{\lambda}-}^N(a)) \geq \lambda$. 
Therefore, by the reset mechanism of \eqref{eq:finitemodel2},
$$
U_{T_n^{\lambda}}^N(a)=
0, \quad \text{ if } \quad \Phi(U_{T_n^{\lambda}-}^N(a)) \geq \lambda
$$
and 
$$
|U_{T_n^{\lambda}}^N(a)|\leq |U_{T_n^{\lambda}-}^N(a)|, \quad  \text{ if } \quad  \Phi(U_{T_n^{\lambda}-}^N(a)) < \lambda.
$$
In particular, this implies that
$$
|U_{T_n^{\lambda}}^N(a)| \leq \sup\{r\geq 0:\Phi(r)<\lambda\}.
$$
In the case $\lambda=M^>(l)$, by the definition of $M^>(l)$, note that $\Phi(r)\geq \lambda$ for any $r>l$. This implies that
$$
|U_{T_n^{M^>(l)}}^N(a)| \leq \sup\{r\geq 0:\Phi(r)<M^>(l)\} \leq l.
$$
\end{proof}

We conclude that the homogeneous Poisson process with rate $M^>(l)$ given by
\begin{equation} \label{eq:homoppupper}
(\mathbf{N}^a([0,t]\times [0,M^>(l)]):t\geq 0)
\end{equation}
is related with the jumps of the actor $a \in \A$ in the following way: for each mark of this process, the actor $a$ tries to jump. This means that if its social pressure is
greater or equal than $l$, the actor jumps. This implies that, at each mark of this Poisson process, either its social pressure is reset to zero or its social pressure in absolute value is bounded by $l$. This is exactly the content of Proposition \ref{prop:bigjumpbound}.


The next proposition improves the bounds for the social pressure presented in Proposition \ref{bound1} by considering Proposition \ref{prop:bigjumpbound}.

\begin{proposition} \label{bound2}
Assume that $M^<(2h)<\infty$ and suppose that for a given  $l\geq0$ we have that $M^>(l)>0$. Let $a\in \A$ and for $t>0$,
let 
$$
\tau =t-\sup\{s\in[0,t]:\mathbf{N}^a([s,t]\times [0,M^>(l)]\geq 1\},
$$
with the convention that $\sup\{\emptyset\}=0$.
Then there exists a Poisson process $(P_s(N\times M^<(2h)))_{s\geq 0}$ with rate $N\times M^<(2h)$, independent of $\mathbf{N}^a ,$ such that
the following inequalities hold (the first inequality holds for the usual
stochastic order).
    \begin{enumerate}
        \item 
   $ \displaystyle
|U_t^N(a)| \leq |U_0^N(a)| \mathbf{1}\{\tau = t\}+l+h+\frac{2h}{N}\times P_{\tau}(N \times M^<(2h)). 
$
\item
$
\E|U_t^N(a)| \leq \E|U_0^N(a)| e^{-M^>(l)t}+l+h+2hM^<(2h)M^>(l).
$
\end{enumerate}
\end{proposition}

\begin{proof}
For any $0\leq r<t$, it follows from \eqref{eq:finitemodel2} that
\begin{multline*}
|U_t^{N}(a)|\leq |U_r^{N}(a)|-\int_r^t \int_0^{+\infty} |U_{s-}^{N}(a)|\mathbf{1}\{z \leq \Phi(U_{s-}^{N}(a))\}\mathbf{N}^{a}(ds,dz)
\\
+\frac{h}{N}\sum_{b\neq a}\int_r^t \int_0^{+\infty} \left(
\mathbf{1}\{z \leq \phi(U_{s-}^{N}(b))\}-\mathbf{1}\{\phi(U_{s-}^{N}(b))<z \leq \Phi(U_{s-}^{N}(b))\}\right)\mathbf{N}^{b}(ds,dz).
\end{multline*}
This implies that
\begin{equation}\label{eq:eqresetbound1}
|U_t^{N}(a)|\leq |U_r^{N}(a)|
+ \frac{h}{N}\sum_{b\neq a}(\tilde{Z}_t(b) - \tilde{Z}_{r}(b)).
\end{equation}
Note that 
$\tau$ is the distance between $t$ and the time of the last mark of the homogeneous Poisson process \eqref{eq:homoppupper}, if there is a mark. If there is no mark, then $\tau =t$
(see Figure \ref{fig1}).
If $\tau<t$, then $t-\tau$ is the time of a mark of the Poisson process \eqref{eq:homoppupper}. Therefore, 
by Proposition \ref{prop:bigjumpbound},
$$
|U_{t-\tau}^N(a)|
\begin{cases}
=|U_0^N(a)|, &\text{ if } \tau=t, \\  
\leq l, &\text{ if } \tau <  t.
\end{cases}
$$

By taking $r=t-\tau $ in \eqref{eq:eqresetbound1} we have that
$$
|U_t^{N}(a)|\leq |U_0^N(a)|\mathbf{1}\{\tau= t\}+l+\frac{h}{N}\sum_{b\neq a}(\tilde{Z}_t(b) - \tilde{Z}_{t-\tau}(b)).
$$
By putting this together with Remark \ref{remark:zttilde} and noting that $\tau$ is independent of $(\mathbf{N}^b(ds,dz):b\neq a)$, we conclude the proof of Part 1 of Proposition \ref{bound2}. By using again the independence of $\tau$ and $(\mathbf{N}^b(ds,dz):b\neq a)$ and by noting that $\tau \sim Exp(M^>(l)) \wedge t$, where $Exp(M^>(l))$ is an exponential random variable with mean $M^>(l)$,
Part 2 follows directly.

\end{proof}

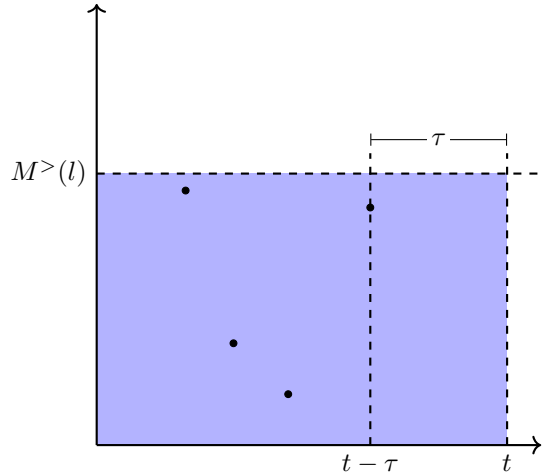
\begin{figure}[h]
\centering
\begin{tikzpicture}[scale=0.9]
\fill[blue!30!white] (0,0) rectangle (6,4);
\draw[thick,->] (0,0) -- (6.5,0);
\draw[thick,->] (0,0) -- (0,6.5);

\draw[thick,dashed] (0,4)  node[anchor=east] {$M^>(l)$} -- (6.5,4);
\draw[thick,dashed] (6,0)  node[anchor=north] {$t$} -- (6,4.3);

\draw[thick,dashed] (4,0)  node[anchor=north] {$t-\tau$} -- (4,4.3);
\draw [fill] (4,3.5) circle [radius=0.05];
\draw [fill] (2,1.5) circle [radius=0.05];
\draw [fill] (1.3,3.75) circle [radius=0.05];
\draw [fill] (2.8,0.75) circle [radius=0.05];

\draw[|-|] (4,4.5) -- node[fill=white,inner sep=0.5mm,outer sep=0.5mm] {$\tau$} (6,4.5);

\end{tikzpicture}
\caption{This figure represents $\tau$ defined in Proposition \ref{bound2}. The points in the blue area are the marks of the Poisson process $(\mathbf{N}^a([0,s]\times [0,M^>(l)]):s\geq 0)$ until time $t$. $\tau$ is the distance between $t$ and the time of the last mark of this homogeneous Poisson process, if there is a mark. }
\label{fig1}
\end{figure}


Now we can prove the main theorem of this section.
\begin{proof}[Proof of Theorem \ref{teo:bound}]
On the event $\{E=r, Z=z\}$ (see Figure \ref{fig2}), by the definitions of $E$ and $Z$ we have that  the Poisson process
$$
(\mathbf{N}^a([0,s]\times [0,r]):s\geq 0)
$$
has a mark at time $t-z$. Therefore, for any $l\geq 0$ such that $M^>(l)\geq r$,
the Poisson process
$$
(\mathbf{N}^a([0,s]\times [0,M^>(l)]):s\geq 0)
$$
has a mark at time $t-z$. Note that, since by assumption we have that $\lim_{l\to \infty}M^>(l)=\infty$, it follows that $\{l\geq 0: M^>(l)\geq r\} \neq \emptyset$, for any $r\geq 0$. By Proposition \ref{prop:bigjumpbound}, this implies that
$$
|U_{t-z}^N(a)|\leq l, \text{ for any } l\geq 0 \text{ such that } M^>(l)\geq r.
$$
This implies that
$$
|U_{t-z}^N(a)|\leq \inf\{l\geq 0: M^>(l) \geq r\}=(M^>)^{-1}(r). 
$$
Since this holds for any $z$ and any $r$, we conclude that
$$
|U_{t-Z}^N(a)|\leq (M^>)^{-1}(E).
$$

By considering \eqref{eq:eqresetbound1} with $r=t-Z$, we have that
$$
|U_t^N(a)|\leq (M^>)^{-1}(E)+\frac{h}{N}\sum_{b\neq a}(\tilde Z_t(b)-\tilde Z_{t-Z}(b)).
$$
By putting this together with Remark \ref{remark:zttilde} and noting that $E$ and $Z$ are independent of $(\mathbf{N}^b(ds,dz):b\neq a)$, we conclude the proof of Part 1 of Theorem \ref{teo:bound}. By using again the independence of $E$, $Z$ and $(\mathbf{N}^b(ds,dz):b\neq a)$ and by noting that $E \sim Exp(t-s)$ and $Z \sim Unif(0,t-s)$, where $Exp(t-s)$ is an exponential random variable with mean $t-s$  and $Unif(0,t-s)$ is an uniform random variables in $[0,t-s]$, Part 2 follows directly.
\end{proof}

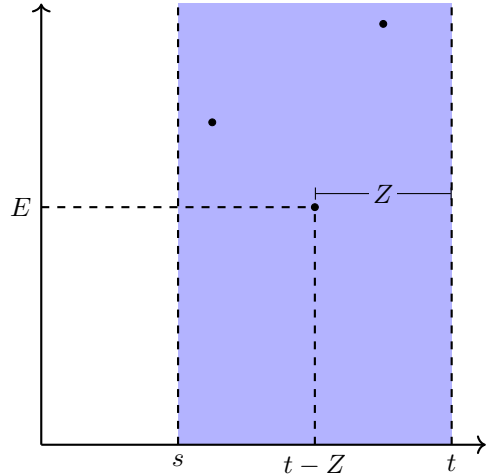
\begin{figure}[h]
\centering
\begin{tikzpicture}[scale=0.9]
\fill[blue!30!white] (2,0) rectangle (6,6.5);
\draw[thick,->] (0,0) -- (6.5,0);
\draw[thick,->] (0,0) -- (0,6.5);

\draw[thick,dashed] (2,0)  node[anchor=north] {$s$} -- (2,6.5);
\draw[thick,dashed] (6,0)  node[anchor=north] {$t$} -- (6,6.5);

\draw[thick,dashed] (0,3.5)  node[anchor=east] {$E$} -- (4,3.5);
\draw[thick,dashed] (4,0)  node[anchor=north] {$t-Z$} -- (4,3.5);
\draw [fill] (4,3.5) circle [radius=0.05];

\draw [fill] (2.5,4.75) circle [radius=0.05];
\draw [fill] (5,6.2) circle [radius=0.05];

\draw[|-|] (4,3.7) -- node[fill=blue!30!white,inner sep=0.5mm,outer sep=0.5mm] {$Z$} (6,3.7);

\end{tikzpicture}
\caption{This figure represents $Z$ and $E$ defined in Theorem \ref{teo:bound}. The points in the blue area are the first three marks of the Poisson process $(\mathbf{N}^a([s,t]\times [0,r]):r\geq 0)$. $E$ is the height of the first mark of this homogeneous Poisson process, and $Z$ is distance between $t$ and this mark.}
\label{fig2}
\end{figure}

\begin{remark}
Note that Theorem \ref{teo:bound} gives us a bound on $|U_t^N(a)|$ that is independent of the initial condition. Note that this bound depends on the choice of $t-s \in (0,t)$. The first term of the bound
$$
(M^>)^{-1}(E) 
$$
decreases when $(t-s)$ increases and the second term
$$
\frac{2h}{N}P_{Z}(N\times M^<(2h))
$$
increases when $(t-s)$ increases. Part 2 of Theorem \ref{teo:bound} suggests that for sufficiently large values of $t>0$ we should take a value of $s \in (0,t)$ that provides a balance of the two terms.
\end{remark}

\begin{remark}
Note that depending on $\phi$ and $(t-s)$, $\E((M^>)^{-1})(E))$ can be infinite. Consider for example $\Phi(x)=\ln(1+x)$. In this case, $M^>(x)=\Phi(x)$ and 
$$
(M^>)^{-1}(x)=e^x-1.
$$
We have that
$$
\E((M^>)^{-1}(E))= \int_0^{\infty}(e^x-1)(t-s)e^{-(t-s)x}dx.
$$
Therefore,
$$
\E((M^>)^{-1}(E))=
\begin{cases}
\frac{1}{(t-s)-1}, &\text{ if } (t-s) >1, \\
\infty, &\text{ if } (t-s) \leq 1.    
\end{cases}
$$
This implies that Part 2 of Theorem \ref{teo:bound} in the case $\Phi(x)=\ln(1+x)$ and $t-s\leq 1$ is a trivial inequality.
\end{remark}

\section{Proof of Theorem \ref{teo:tight}} \label{sec:tight}

For the remainder of this article we will use the representation of $(U_t^N(a):a\in \A, t\geq 0)$ given by \eqref{eq:finitemodel}.

\begin{proof}[Proof of Theorem \ref{teo:tight}]
It is well known that Part 2 follows from Part 1 and the exchangeability of the system, see \cite{sznit} [Proposition 2.2(ii)].
To show Part 1 we use the criterion of
Aldous, see \cite{jacshi} [Theorem 4.5, page 356]. It is sufficient to prove that


\begin{itemize}
\item[a)] 
For any $T>0$,
$$
\lim_{K\to \infty}\sup_{N\geq 2}\P\left(\sup_{t\leq T}|U_t^{N}(1)| > K\right)=0.
$$
\item[b)] For all $T>0$, for all $\epsilon>0$,
$$
\lim_{\delta \to 0}\lim_{N \to \infty} \sup_{(S,S')\in A_{\delta, T}} \P(|U_{S'}^{N}(1)-U_{S}^{N}(1)|> \epsilon)=0,
$$
where $A_{\delta, T}$ is the set of all pairs of stopping times $(S,S')$ such that $0\leq S \leq S'\leq S+\delta \leq T$.
\end{itemize}

Item a) follows directly from Proposition \ref{bound1}. 
For item b), if follows from \eqref{eq:finitemodel} that
$$
U_{S'}^{N}(1)-U_{S}^{N}(1)=
-\sum_{o \in \mathcal{O}}\int_S^{S'} \int_0^{+\infty} U_{s-}^{N}(1)\mathbf{1}\{z \leq \phi(o U_{s-}^{N}(1))\}\mathbf{N}^{1, o}(ds,dz)
$$
$$
+ \frac{h}{N}\sum_{o \in \mathcal{O}}\sum_{b\neq 1}\int_S^{S'} \int_0^{+\infty} 
o\mathbf{1}\{z \leq \phi(oU_{s-}^{N}(b))\}\mathbf{N}^{b,o}(ds,dz).
$$
Therefore,
$$
|U_{S'}^{N}(1)-U_{S}^{N}(1)|\leq 
\sum_{o \in \mathcal{O}}\int_S^{S'} \int_0^{+\infty} |U_{s-}^{N}(1)|\mathbf{1}\{z \leq \phi(o U_{s-}^{N}(1))\}\mathbf{N}^{1, o}(ds,dz) +\frac{h}{N}(Z_{S'}^N-Z_S^N).
$$
Moreover, note that for any $\epsilon>0$,
\begin{multline*}
\P\left(\sum_{o \in \mathcal{O}}\int_S^{S'} \int_0^{+\infty} |U_{s-}^{N}(1)|\mathbf{1}\{z \leq \phi(o U_{s-}^{N}(1))\}\mathbf{N}^{1, o}(ds,dz)>\epsilon\right)\leq \\
\P\left(\sum_{o \in \mathcal{O}}\int_S^{S'} \int_0^{+\infty} \mathbf{1}\{z \leq \phi(o U_{s-}^{N}(1))\}\mathbf{N}^{1, o}(ds,dz)\geq 1\right). 
\end{multline*}
Therefore, 
\begin{multline*}
\P(|U_{S'}^{N}(1)-U_{S}^{N}(1)|> \epsilon) \leq 
\\
\P\left(\sum_{o \in \mathcal{O}}\int_S^{S'} \int_0^{+\infty} \mathbf{1}\{z \leq \phi(o U_{s-}^{N}(1))\}\mathbf{N}^{1, o}(ds,dz)\geq 1\right)+ \P\left(\frac{h}{N}(Z_{S'}^N-Z_S^N) >\frac{\epsilon}{2}\right).
\end{multline*}
Putting all this together with the
Markov inequality and recalling that $S'-S<\delta$, we conclude that the left-hand side of the inequality above is upper bounded by
\begin{equation}
\label{eq:tight}
\E\left(\sum_{o\in \opn}\int_S^{S+\delta} \int_0^{+\infty} \mathbf{1}\{z \leq \phi(oU_{s-}^{N}(1))\}\mathbf{N}^{1,o}(ds,dz)\right)+\frac{2h}{N\epsilon}\E(Z_{S+\delta}^N-Z_S^N).
\end{equation}
For any $l>0$, the first term of \eqref{eq:tight} is equal to 
$$
\E\left(\int_S^{S+\delta} \Phi(U_{s}^{N}(1))\mathbf{1}\{|U_{s}^{N}(1)|\leq l\}ds\right)+\E\left(\int_S^{S+\delta} \Phi(U_{s}^{N}(1))\mathbf{1}\{|U_{s}^{N}(1)|> l\}ds\right).
$$
This implies that for any $l>0$, the first term of \eqref{eq:tight} is upper bounded by
$$
\delta M^<(l)+\frac{1}{l}\E\left(\int_0^{T} |U_s^N(1)|\Phi(U_{s}^{N}(1))ds\right).
$$
Recall that by Assumption \ref{as:phi1},  $M^<(r)< \infty$ for any $r\geq 0$.
If $M^<(r)\to \infty$ as $r\to \infty$,
by considering the inequality above with $l=l(\delta)=\frac{1}{2}\sup\{r>0:M^<(r)\leq \delta^{-1/2}\}$ and using Remark \ref{coro1} we conclude that the first term of \eqref{eq:tight} goes to $0$ as $\delta \to 0$. In the other case, that is, if $\phi$ is bounded, it is sufficient to take $l=l(\delta)=\delta^{-1/2}$ to conclude that the first term of \eqref{eq:tight} goes to $0$ as $\delta \to 0$.

Finally, by exchangeability of the system note that
\begin{multline*}
\E(Z_{S+\delta}^N-Z_S^N)=N\E(Z_{S+\delta}^N(1)-Z_S^N(1))=
\\
N\E\left(\sum_{o\in \opn}\int_S^{S+\delta} \int_0^{+\infty} \mathbf{1}\{z \leq \phi(oU_{s-}^{N}(1))\}\mathbf{N}^{1,o}(ds,dz)\right).
\end{multline*}
Therefore,  the second term of \eqref{eq:tight} can be treated as the first term, and it goes to $0$ as $\delta \to 0$, which concludes the proof.



\end{proof}

\section{Proof of Theorem \ref{teo:sde}} \label{sec:sde}

We first prove Part 1 of Theorem \ref{teo:sde}.
\begin{proof}
Consider two solutions $(U_t)_{t\geq 0}$ and $(\tilde U_t)_{t\geq 0}$ of (2), driven by the same Poisson measures $\mathbf{N}^{-1}$, $\mathbf{N}^{+1}$ and with $U_0=\tilde U_0$:
\begin{eqnarray*}
U_t&=&U_0-\sum_{o \in \opn}\int_0^t \int_0^{+\infty} U_{s-}\mathbf{1}\{z \leq \phi(oU_{s-})\}\mathbf{N}^o(ds,dz)
+h\sum_{o\in \mathcal{O}}o \int_0^t \E[\phi(oU_{s})]ds,
\\
\tilde U_t&=&U_0-\sum_{o \in \opn}\int_0^t \int_0^{+\infty} \tilde U_{s-}\mathbf{1}\{z \leq \phi(o\tilde U_{s-})\}\mathbf{N}^o(ds,dz)
+h\sum_{o\in \mathcal{O}}o \int_0^t \E[\phi(o\tilde U_{s})]ds.
\end{eqnarray*}
We have that 
\begin{multline*}
|U_t-\tilde U_t|\leq 
\sum_{o\in \opn}\Big[\int_0^t \int_0^{+\infty}-|U_{s-}-\tilde U_{s-}|\mathbf{1}\{z \leq \phi(o\tilde U_{s-})\wedge \phi(oU_{s-}) \} + 
\\
|\tilde U_{s-}|\vee |U_{s-}|\mathbf{1}\{\phi(o\tilde U_{s-})\wedge \phi(oU_{s-})< z \leq \phi(o\tilde U_{s-})\vee \phi(oU_{s-}) \}\mathbf{N}^o(ds,dz)\Big]+
\\
h\int_0^t |\E(\phi(U_s))-\E(\phi(-U_s))-\E(\phi(\tilde U_s))+\E(\phi(-\tilde U_s))|ds.
\end{multline*}
Therefore,
\begin{multline*}
\E\left(\sup_{s\leq t}|U_s-\tilde U_s|\right)\leq 
\int_0^t \E\left[(|\tilde U_{s}| \vee |U_{s}|)\left(\sum_{o\in \opn}|\phi(o\tilde U_{s})-\phi(oU_{s})|\right)\right]ds
\\
+h\int_0^t |\E(\phi(U_s))-\E(\phi(\tilde U_s))|+|\E(\phi(-U_s))-\E(\phi(-\tilde U_s))|ds.
\end{multline*}
Recalling Remark \ref{remark:ctdef}, by Proposition \ref{boundlimit1} and Assumption \ref{as:boundedandcont}, we know that
$$
\sup\{|U_{s}|,|\tilde U_{s}|: s \in [0,t]\} \leq 2L+2htM^<(2h) = \gamma_t.
$$
This implies that
$$
\max\{\Phi(U_{s}),\Phi(\tilde U_{s}): s \in [0,t]\} \leq M^<(\gamma_t).
$$
Note that $M^<(\gamma_t) <\infty$, by Assumption \ref{as:phi2}.
Since $\phi$ is locally Lipschitz by Assumption \ref{as:phi2}, there exists $C_{lip}(\gamma_t)>0,$ the Lipschitz constant of  $\phi$ restricted to $[-\gamma_t,\gamma_t],$ such that for $s\leq t$ and for any $o\in \opn$,
$$
|\phi(o\tilde U_{s})-\phi(oU_{s})| \leq C_{lip}(\gamma_t)|\tilde U_{s}-U_{s}|
$$
and
$$
|\E(\phi(oU_s))-\E(\phi(o\tilde U_s))|\leq C_{lip}(\gamma_t)\E(|\tilde U_{s}-U_{s}|).
$$ 

We conclude that 
$$
\E\left(\sup_{s\leq t}|U_s-\tilde U_s|\right)\leq  [2\gamma_t C_{lip}(\gamma_t)+2hC_{lip}(\gamma_t)] \int_0^t\E(|U_s-\tilde U_s|)ds.
$$
The assertion then follows from  Grönwall's lemma.
\end{proof}

The proof of part 2 of Theorem \ref{teo:sde} follows from a classical Picard iteration and is given now.

\begin{proof}

\textbf{Step 1.} Here we suppose that $\phi:\R\to \R^+$ is bounded and globally Lipschitz. In Step 2 we will relax this hypothesis. Let us denote $C_{lip}>0$ the Lipschitz constant and $K>0$ the bound such that for all $x,y \in \R$,
$$
|\phi(x)-\phi(y)|\leq C_{lip}|x-y|,
$$
and
$$
|\phi(x)|\leq K .
$$

We consider a classical Picard iteration where
$U^{[0]}_t\equiv U_0$, with $|U_0|\leq L$, and for any $n\geq 1$,
$$
U_t^{[n]}=U_0-\sum_{o\in \opn}\int_0^t \int_0^{+\infty} U_{s-}^{[n]}\mathbf{1}\{z \leq \phi(oU_{s-}^{[n]})\}\mathbf{N}^o(ds,dz)
+h\sum_{o\in \mathcal{O}}o \int_0^t \E[\phi(oU_{s}^{[n-1]})]ds.
$$
Let us fix  $T>0$. Then,
\begin{equation}\label{eq:prioriboundpicard}
|U_t^{[n]}|\leq L+2hKT,
\end{equation}
for any $t\leq T$ and for any $n\geq 1$.
Therefore,
for any $t>0$ such that $t\leq T$, following the steps of the proof of part 1 of Theorem \ref{teo:sde}, we have that
\begin{multline*}
    \E\left(\sup_{s\leq t}|U_s^{[n]}-U_s^{[n-1]}|\right)\leq 
(L+2hKT)\sum_{o\in \opn}   \int_0^t \E(|\phi(oU_s^{[n]})-\phi(oU_s^{[n-1]})|)ds
\\+h\sum_{o\in \opn}   \int_0^t\E(|(\phi(oU_s^{[n-1]})-\phi(oU_s^{[n-2]})|)ds
\\
\leq 
C(T)\int_0^t \E\left(\sup_{r\leq s}|U_r^{[n]}-U_r^{[n-1]}|\right)ds+2hC_{lip}\int_0^t \E\left(\sup_{r\leq s}|U_r^{[n-1]}-U_r^{[n-2]}|\right)ds,
\end{multline*}
with $C(s)=2(L+2hKs)C_{lip}$, for any $s\geq 0$.

By Grönwall's lemma, we conclude that 
\begin{equation}\label{eq:picardrec}
\E\left(\sup_{s\leq t}|U_s^{[n]}-U_s^{[n-1]}|\right)\leq c(t) \E\left(\sup_{s\leq t}|U_s^{[n-1]}-U_s^{[n-2]}|\right),
\end{equation}
where
$
c(t):=2hC_{lip}e^{tC(T)}.
$
Since $c:[0,T]\to \R^+$ is an increasing and continuous function satisfying $c(0)=0$, there exists $t^* \in (0, T ) $ such that $c(t^*)<1$. The a priori bound \eqref{eq:prioriboundpicard} implies that
$$
\E\left(\sup_{s\leq t^*}|U_s^{[1]}-U_s^{[0]}|\right) \leq 2L+2hKT.
$$
Therefore, by iterating \eqref{eq:picardrec}, we conclude that
$$
\E\left(\sup_{s\leq t^*}|U_s^{[n]}-U_s^{[n-1]}|\right) \leq (c(t^*))^n(2L+2hKT).
$$
In particular, 
$$
\E\left(\sum_{n=1}^{\infty}\sup_{s\leq t^*}|U^{[n]}_s-U^{[n-1]}_s| \right) < \infty ,
$$
and therefore, putting
$$
\bar{U}_s=U^{[0]}_s+\sum_{n=1}^{\infty}(U^{[n]}_s-U^{[n-1]}_s) ,
$$
we have that
$$
U^{[n]}\to \bar{U}, \text{ as } n \to \infty,
$$
uniformly in $[0,t^*]$ both a.s and in $L^1$. 

In the following we will show that $(\bar{U}_t)_{t\in [0,t^*]}$ is indeed solution of the limit equation \eqref{eq:limiteq}. First note that, for any $t\leq t^*$ and any $o \in \opn$,
$$
\left|\int_0^t \E[\phi(oU_{s}^{[n]})]ds - \int_0^t \E[\phi(o\bar{U}_{s})]ds\right| \leq C_{lip}\int_0^t \E[|U_{s}^{[n]}-\bar{U}_{s}|]ds.
$$ 
By the a priori bound \eqref{eq:prioriboundpicard} and by the convergence of $(U_s^{[n]})_{s\leq t^*}$ to $(\bar{U}_s)_{s\leq t^*}$ in $L^1$, we conclude that 
\begin{equation}\label{eq:driftconvergencepicard}
\int_0^t\E[|U_{s}^{[n]}-\bar{U}_{s}|]ds \to 0, \text{ as } n\to \infty.
\end{equation}

Moreover, by the a priori bound \eqref{eq:prioriboundpicard}, using similar arguments as those above, it follows that
\begin{multline*}
    \E\left(\left|\sum_{o\in \opn}\int_0^t \int_0^{+\infty}\left( U_{s-}^{[n]}\mathbf{1}\{z \leq \phi(oU_{s-}^{[n]})\}-  \bar{U}_{s-}\mathbf{1}\{z \leq \phi(o \bar{U}_{s-})\}\right)\mathbf{N}^o(ds,dz)\right|\right) \leq 
\\
\tilde{C}(t)\int_0^t\E(|U_{s}^{[n]}-\bar{U}_{s}|)ds \to 0, \text{ as } n\to \infty ,
\end{multline*}
for some $\tilde{C}(t)>0$. Taking a subsequence $(n_k)_{k\geq 1}$ to obtain almost sure convergence, we have that for any $o\in \opn$,
\begin{equation}
\label{eq:picardlimit2}
\int_0^t \int_0^{+\infty} U_{s-}^{[n_k]}\mathbf{1}\{z \leq \phi(oU_{s-}^{[n_k]})\}\mathbf{N}^o(ds,dz)\to \int_0^t \int_0^{+\infty} \bar{U}_{s-}\mathbf{1}\{z \leq \phi(oU_{s-})\}\mathbf{N}^o(ds,dz),
\end{equation}
a.s. as $k\to \infty$.
Putting together \eqref{eq:driftconvergencepicard} and \eqref{eq:picardlimit2} we conclude that $(\bar{U}_t)_{t\in [0,t^*]}$ satisfies \eqref{eq:limiteq}. 

Once the convergence is proven in the time interval $[0,t^*]$, we can proceed by iteration
over successive intervals $[kt^*, (k + 1)t^*]$ to conclude that $(\bar{U}_t)_{t\in [0,T]}$ is solution of the limit equation \eqref{eq:limiteq} in $[0,T]$. Note that for any $T>0$, we can find a solution of \eqref{eq:limiteq} in $[0,T]$, that we will denote in the following as $(\bar{U}_t^{\{T\}})_{t\in [0,T]}$. Moreover, by the pathwise uniqueness, for any $T<T'$,
$\bar{U}_t^{\{T\}}=\bar{U}_t^{\{T'\}}$ for all $t \in [0,T]$.
Therefore, 
$$
U_t=\sum_{k=1}^{+\infty}\bar{U}_t^{\{k\}}\mathbf{1}\{k-1\leq t < k\}
$$
is a solution of \eqref{eq:limiteq} in $\R^+$.

\textbf{Step 2.} Now, consider $\phi$ satisfying only Assumption \ref{as:phi2} ($\phi$ is locally Lipschitz). 
Fix a time horizon  $T>0$. Then, recalling once more Remark \ref{remark:ctdef}, by Proposition \ref{boundlimit1} and Assumption \ref{as:boundedandcont}, we know that a priori
$$
\sup\{|U_{s}|: s \in [0,T]\} \leq  \gamma_T =: r .
$$  
Let
$$
\tilde{\phi}^r(x)=
\begin{cases}
\phi(x), &\text{ if }|x|\leq r, \\
\phi(r), &\text{ if }x> r, \\
\phi(-r), &\text{ if }x< -r.
\end{cases}
$$

Since $\phi$ is locally Lipschitz by Assumption \ref{as:phi2}, we have that $\tilde{\phi}^r$ is globally Lipschitz and bounded. Let $C_{lip}^r>0$ denote this Lipschitz constant. 
For any $r>0$, by Step 1 there exists a unique strong solution for
$$
\tilde{U}_t^r=\tilde{U}_0^r-\sum_{o\in \opn}\int_0^t \int_0^{+\infty} \tilde{U}_{s-}^r\mathbf{1}\{z \leq \tilde{\phi}^r(o\tilde{U}_{s-}^r)\}\mathbf{N}^o(ds,dz)
+h\sum_{o\in \mathcal{O}}o \int_0^t \E[\tilde{\phi}^r(o\tilde{U}_{s}^r)]ds.
$$
By Proposition \ref{boundlimit1} and recalling Remark \ref{remarkphitilde}, for any $s\leq T$ and for any $o\in \opn$,
$$
\tilde{\phi}^{r}(o\tilde{U}_s^{r})=\phi(o\tilde{U}_s^{r}).
$$
Therefore, $(\tilde{U}_t^{r})_{t\leq T}$ is a solution of \eqref{eq:limiteq} (with the original function $\phi:\R\to\R^+)$ in $[0,T]$. Since this argument is valid for all $T>0$, we have proven the unique existence of \eqref{eq:limiteq} in $\R^+$. 

\end{proof}

\section{Proof of Theorem \ref{teo:convergence}} \label{sec:convergence}
We recall that by \eqref{eq:finitemodel}, for any $a\in \A$ and for any $N\geq 2$,
\begin{multline*}
    U_t^{N}(a)=U_0^{N}(a)-\sum_{o \in \mathcal{O}}\int_0^t \int_0^{+\infty} U_{s-}^{N}(a)\mathbf{1}\{z \leq \phi(o U_{s-}^{N}(a))\}\mathbf{N}^{a, o}(ds,dz)
\\
+ \frac{h}{N}\sum_{o \in \mathcal{O}}\sum_{b\neq a}\int_0^t \int_0^{+\infty} 
o\mathbf{1}\{z \leq \phi(oU_{s-}^{N}(b))\}\mathbf{N}^{b,o}(ds,dz),
\end{multline*}
where $(\mathbf{N}^{a,o}(ds,dz): a\in \A)$ is an i.i.d. family of Poisson random measures  on $\R^+ \times \R^+$ having intensity $dsdz$. Grant Assumptions \ref{as:phi2} and \ref{as:boundedandcont}. For any $a\in \A$, let
\begin{equation}\label{eq:limitmodelcoupling}
U_t(a)=U_0^N(a)-\sum_{o \in \opn}\int_0^t \int_0^{+\infty} U_{s-}(a)\mathbf{1}\{z \leq \phi(U_{s-}(a))\}\mathbf{N}^{a,o}(ds,dz)
+h\sum_{o\in \mathcal{O}}o \int_0^t \E[\phi(oU_{s}(a))]ds.
\end{equation}
In other words,
we construct together the finite model with $N$ social actors and $N$ copies of the limit equation using the same family of Poisson random measures and the same family of i.i.d. initial distributions. 

Now we can prove Theorem \ref{teo:convergence}. 

\begin{proof}[Proof of Theorem \ref{teo:convergence}] 
We work on $ [0, T ] $ for some fixed $ T > 0.$ Let us start by calculating $|U_t^N(a)-U_t(a)|$, for $t \le T,$ defined in \eqref{eq:finitemodel} and \eqref{eq:limitmodelcoupling}. We have that
\begin{eqnarray}\label{eq:utnminusut}
&&|U_t^N(a)-U_t(a)|\leq   
\\
&&\sum_{o\in \opn}\Big[\int_0^t \int_0^{+\infty}-|U_{s-}^N(a)-U_{s-}(a)|\mathbf{1}\{z \leq  \phi(oU_{s-}^N(a)) \wedge \phi(oU_{s-}(a))\}  \nonumber 
\\
&& +|U_{s-}^N(a)|\vee |U_{s-}(a)| (|\mathbf{1}\{z\leq  \phi(oU_{s-}^N(a))\}
-\mathbf{1}\{z\leq  \phi(oU_{s-}(a))\}|)\Big] \mathbf{N}^{a,o}(ds,dz)\nonumber 
\\
&&+ \left|\frac{h}{N}\sum_{o \in \mathcal{O}}\sum_{b\neq a}\int_0^t \int_0^{+\infty} 
o\mathbf{1}\{z \leq \phi(oU_{s-}^{N}(b))\}\mathbf{N}^{b,o}(ds,dz)-h\int_0^t \E[\phi(U_s(a))-\phi(-U_s(a))]ds\right|.\nonumber
\end{eqnarray} 
To deal with the last term on the right-hand side of \eqref{eq:utnminusut}, 
we note that
\begin{multline} \label{eq:driftdecomposition}
h\int_0^t \E[\phi(U_s(a))-\phi(-U_s(a))]ds=
\\ \frac{h}{N}\sum_{b\in \A}\sum_{o\in \opn}\left[\int_0^t \int_0^{\infty} o\mathbf{1}\{z\leq \phi(oU_{s-}(b))\} \mathbf{N}^{b,o}(ds,dz)-
\int_0^t \int_0^{\infty}o\mathbf{1}\{z\leq \phi(oU_{s-}(b))\}\tilde{\mathbf{N}}^{b,o}(ds,dz)\right]
\\ +h\int_0^t \E[\phi(U_s(a))-\phi(-U_s(a))]ds-\frac{h}{N}\sum_{b\in \A}\sum_{o\in \opn}\int_0^t o\phi(oU_{s}(b))ds,
\end{multline}
where for each $b\in \A$ and $o\in \opn$,
$\tilde{\mathbf{N}}^{b,o}(ds,dz)=\mathbf{N}^{b,o}(ds,dz)-dsdz$. 

Due to the exchangeability of the system, which is a consequence of Assumption \ref{as:boundedandcont}, we have 
\begin{equation}\label{eq:hNvtbound}
\left|h\int_0^t \E(\phi(U_s(a))-\phi(-U_s(a))ds-\frac{h}{N}\sum_{b\in \A}\sum_{o\in \opn}\int_0^t o\phi(oU_{s}(b))ds\right| \leq \frac{h}{N}\times |V_t|,
\end{equation}
where
\begin{equation} \label{eq:vtdefinition}
V_t\mydef \sum_{b\in \A}\sum_{o\in \opn}\int_0^t o[\E(\phi(oU_s(b)))-\phi(oU_{s}(b))]ds.
\end{equation}
By putting together \eqref{eq:driftdecomposition} and \eqref{eq:hNvtbound} and by denoting
\begin{equation} \label{eq:martingaledefinition}
M_t \mydef\sum_{b\in \A}\sum_{o\in \opn}\int_0^t \int_0^{\infty}o\mathbf{1}\{z\leq \phi(oU_{s-}(b))\}\tilde{\mathbf{N}}^{b,o}(ds,dz),
\end{equation}
the last term of the right-hand side of \eqref{eq:utnminusut} is therefore upper bounded by
\begin{equation} \label{eq:driftdecomposition2}
\frac{h}{N}\left|\sum_{b\in \A}\sum_{o\in \opn}\int_0^t \int_0^{\infty} o\mathbf{1}\{z\leq \phi(oU_{s-}(b))\} - o
\mathbf{1}\{z \leq \phi(oU_{s-}^{N}(b))\}\mathbf{N}^{b,o}(ds,dz)\right|
\end{equation}
$$
+\left|\frac{h}{N}\sum_{o \in \mathcal{O}}\int_0^t \int_0^{+\infty} 
o\mathbf{1}\{z \leq \phi(oU_{s-}^{N}(a))\}\mathbf{N}^{a,o}(ds,dz)\right|+\frac{h}{N}|V_t|+\frac{h}{N}|M_t|.
$$
Note that the first term in the second line of \eqref{eq:driftdecomposition2} is upper bounded by $\frac{h}{N} \times Z_t^N(a)$, defined in \eqref{eq:ztndefinition}.

Therefore, by putting together \eqref{eq:utnminusut}, \eqref{eq:driftdecomposition} and \eqref{eq:driftdecomposition2} 
we have that
\begin{multline} \label{eq:couplingfinitelimit}
\E\left(\sup_{s\leq t}|U_t^N(a)-U_t(a)|\right)\leq 
\\
\int_0^t \E\left[
|U_{s}^N(a)|\vee |U_{s}(a)|\left(\sum_{o\in \opn} |\phi(oU_{s}^N(a))-\phi(oU_{s}(a))|\right) \right]ds
\\
+\frac{h}{N}\sum_{b\in \A} \int_0^t\E\left[|\phi(U_{s}^N(b))-\phi(U_{s}(b))|+|\phi(-U_{s}^N(b))-\phi(-U_{s}(b))|\right] ds
\\
+ \frac{h}{N}\E(Z_t^N(a))+\frac{h}{N}\E\left(\sup_{s\leq t}|V_s|\right)+\frac{h}{N}\E\left(\sup_{s\leq t}|M_s|\right).
\end{multline}
Now, we analyze each term on the right-hand side of \eqref{eq:couplingfinitelimit}.

Recalling Remark \ref{remark:ctdef}, by Proposition \ref{boundlimit1} and Assumption \ref{as:boundedandcont}, we know that
\begin{equation}\label{eq:ctdefinitionteo3}
\sup\{|U_s(a)|:s\leq t\} \leq 2L+2htM^<(2h) = \gamma_t.
\end{equation}
Note that on the event
$$
\left\{\frac{hZ^N_{s}}{N} \leq l\right\},
$$
for $l>0$, it follows from \eqref{eq:boundsocialandreset} and Assumption \ref{as:boundedandcont} that
$$
|U_{s}^N(a)| \leq L+l.
$$
Recall that $\phi$ is locally Lipschitz due to Assumption \ref{as:phi2}. Denote by $C_{lip}(r)$ the Lipschitz constant of $\phi$ restricted to $[-r,r]$, for any $r>0.$ 

In what follows, we work with a fixed truncation level $ l >0.$ The precise choice of $l, $ depending on $T,$ will be made in \eqref{eq:lstardefinition} below. For the moment being, we upper bound for any fixed $l>0$,
\begin{multline}\label{eq:teo3largedev}
\int_0^t \E\left[
|U_{s}^N(a)|\vee |U_{s}(a)|\left(\sum_{o\in \opn} |\phi(oU_{s}^N(a))-\phi(oU_{s}(a))|\right) \right]ds
 \leq \\
 2[(L+l)\vee \gamma_t]C_{lip}((L+l)\vee \gamma_t)\int_0^t \E\left[|U_{s}^N(a)-U_{s}(a)|\right]ds+ \\
\int_0^t \E\left[
|U_{s}^N(a)|\vee |U_{s}(a)|\left(\sum_{o\in \opn} |\phi(oU_{s}^N(a))-\phi(oU_{s}(a))|\right)\mathbf{1}\left\{\frac{hZ^N_{s}}{N}> l\right\} \right]ds.
\end{multline}
In what follows we will show how to upper bound the last term on the right-hand side of \eqref{eq:teo3largedev}. 
To do so, we will use our a priori upper bounds. We start observing that, for any $s\geq 0$, $|U_{s}^N(a)|\vee |U_{s}(a)|\leq |U_{s}^N(a)|+|U_{s}(a)|$ and 
$$
\sum_{o\in \opn} |\phi(oU_{s}^N(a))-\phi(oU_{s}(a))| \leq \Phi(U_{s}^N(a))+\Phi(U_{s}(a)).
$$ 
Moreover, recall that $\gamma_t$ is an upper bound for $|U_s(a)|$, for any $s\leq t$ and $a\in \A$. This implies that  $M^<(\gamma_t)$ is an upper bound for $\Phi(U_s(a))$, for any $s\leq t$ and $a\in \A$. 
Putting all this together,
we have that
the  term of the last line of \eqref{eq:teo3largedev} is upper bounded by
\begin{equation} \label{eq:teo3_eq2}
\int_0^t \E\left[|U_{s}^N(a)|\Phi(U_{s}^N(a))\mathbf{1}\left\{\frac{hZ_{s}^N}{N}> l\right\}\right]ds+\gamma_t\int_0^t \E\left[\Phi(U_{s}^N(a))\mathbf{1}\left\{\frac{hZ_{s}^N}{N}> l\right\}\right]ds
\end{equation}
$$
+M^<(\gamma_t)\int_0^t \E\left[|U_{s}^N(a)|\mathbf{1}\left\{\frac{hZ_{s}^N}{N}> l\right\}\right]ds+\gamma_t M^<(\gamma_t)t\P\left[\frac{hZ_{t}^N}{N}> l\right].
$$

Now we use the following important trick that allows us to rewrite the first term in \eqref{eq:teo3_eq2} as
$$
\E\left(\sum_{o \in \opn}\int_0^t \int_0^{+\infty} |U_{s}^{N}(a)|\mathbf{1}\{z \leq \phi(o U_{s}^{N}(a))\}\mathbf{1}\left\{\frac{hZ_{s}^N}{N}> l\right\}dsdz\right)=
$$
$$
\E\left(\sum_{o \in \opn}\int_0^t \int_0^{+\infty} |U_{s-}^{N}(a)|\mathbf{1}\{z \leq \phi(o U_{s-}^{N}(a))\}\mathbf{1}\left\{\frac{hZ_{s-}^N}{N}> l\right\}\mathbf{N}^{a, o}(ds,dz)\right).
$$
Therefore, using that $(Z_s^N)_{s\geq 0}$ is increasing and that for any positive random variable $X$, $\E(X)=\int_0^{\infty}\P(X>x)$, we have that
the first term in \eqref{eq:teo3_eq2} is upper bounded by
$$
\E\left(\mathbf{1}\left\{\frac{hZ_{T}^N}{N}> l\right\}\sum_{o \in \opn}\int_0^t \int_0^{+\infty} |U_{s-}^{N}(a)|\mathbf{1}\{z \leq \phi(o U_{s-}^{N}(a))\}\mathbf{N}^{a, o}(ds,dz)\right)=
$$
$$
 \int_0^{\infty}\P\left(\mathbf{1}\left\{\frac{hZ_{T}^N}{N}> l\right\}\left(\sum_{o \in \opn}\int_0^t \int_0^{+\infty} |U_{s-}^{N}(a)|\mathbf{1}\{z \leq \phi(o U_{s-}^{N}(a))\}\mathbf{N}^{a, o}(ds,dz)\right)>x\right)dx.
$$
By \eqref{eq:boundsocialandreset} and Assumption \ref{as:boundedandcont}, 
we have that the last term above is upper bounded by
$$
\int_0^{\infty}\P\left(\mathbf{1}\left\{\frac{hZ_{T}^N}{N}> l\right\}\left(L+h\frac{Z_{T}^N}{N}\right)>x\right)dx =
$$
\begin{equation} \label{eq:teo3_eq25}
 (L+l)\P\left(\frac{hZ_{T}^N}{N}> l\right)+\int_{L+l}^{\infty}\P\left(L+h\frac{Z_{T}^N}{N}>x\right)dx.    
\end{equation}
By Proposition \ref{bound1}, the right-hand side of \eqref{eq:teo3_eq25} is upper bounded by
\begin{equation}
\label{eq:teo3_eq3}
(L+l)\P\left(P_T(N\times M^<(2h))>\frac{N}{2h}(l-h)\right)+\int_{L+l}^{\infty}\P\left(L+h+\frac{2h}{N}\times P_T(N\times M^<(2h))>x\right)dx.
\end{equation}
In what follows, we use that a Poisson random variable $P(\lambda)$ with mean $\lambda>0$ satisfies for any $x>0$ 
$$
\P(P(\lambda)>\lambda+x)\leq e^{-x^2/(2(\lambda+x))}.
$$
By taking 
\begin{equation}\label{eq:lstardefinition}
l=l^*(T)=2hTM^<(2h)+3h , 
\end{equation}
we have that
\begin{equation}\label{eq:poissonboundcnt}
\P\left(\frac{hZ_{T}^N}{N}>l^*(T)\right)
\leq \P\left(P_T(N\times M^<(2h))>\frac{N}{2h}(l^*(T) -h)\right)
\leq e^{-Nh/(l^*(T)-h)} = e^{ - c_T N },
\end{equation}
where $ c_T =  \frac{1}{2(T M^< ( 2h) + 1)}.$
Therefore, the first term in \eqref{eq:teo3_eq3} is upper bounded by \footnote{In the definition of $\epsilon (N, T )$ we added $1$ to the factor $L+l^*(T)$ to ensure that $\epsilon (N, T )$ is also an upper bound for \eqref{eq:poissonboundcnt}. }
\begin{equation}\label{eq:cntdefinition}
\epsilon (N, T ) :=(1+L+l^*(T)) e^{ - c_T N }
\end{equation}
and the second term is upper bounded by 
$$
e^{ - c_T N }
\int_{0}^{\infty}e^{-Nx/(4h(TM^<(2h)+1))}dx=
\frac{2l^*(T)}{N} e^{-Nh/(l^*(T)-h)}\leq \epsilon (N, T ).
$$
With this, we conclude that the  first term of \eqref{eq:teo3_eq2} satisfies 
\begin{equation} \label{eq:firsttermbound}
\int_0^t \E\left[|U_{s}^N(a)|\Phi(U_{s}^N(a))\mathbf{1}\left\{\frac{hZ_{s}^N}{N}> l^* ( T) \right\}\right]ds \leq 2\epsilon (N, T ).
\end{equation}
To upper bound the second term of \eqref{eq:teo3_eq2}, note that
\begin{multline*}
\int_0^t \E\left[\Phi(U_{s}^N(a))\mathbf{1}\left\{\frac{hZ_{s}^N}{N}> l^*(T)\right\}\right]ds \leq
\\
\frac{1}{2h}\int_0^t \E\left[|U_{s}^N(a)|\Phi(U_{s}^N(a))\mathbf{1}\left\{\frac{hZ_{s}^N}{N}> l^*(T)\right\}\right]ds+
\\
\int_0^t \E\left[\Phi(U_{s}^N(a))\mathbf{1}\left\{\frac{hZ_{s}^N}{N}> l^*(T)\right\}\mathbf{1}\left\{|U_{s}^N(a)|\leq 2h\right\}\right]ds.
\end{multline*}
By \eqref{eq:poissonboundcnt} and \eqref{eq:firsttermbound}, we conclude that the right-hand side of the inequality above is upper bounded by 
\begin{equation}\label{eq:secondtermbound}
\frac{1}{h}\epsilon (N, T )+tM^<(2h)\P\left(\frac{hZ_{T}^N}{N}> l^*(T)\right)\leq \epsilon (N, T )\left(\frac{1}{h}+tM^<(2h)\right).
\end{equation}
To upper bound the third term of \eqref{eq:teo3_eq2}, note that, by Proposition \ref{bound1} and recalling \eqref{eq:teo3_eq25}, we have that
$$
\int_0^t \E\left[|U_{s}^N(a)|\mathbf{1}\left\{\frac{hZ_{s}^N}{N}> l^*(T)\right\}\right]ds \leq t\E\left((L+h\frac{Z_{T}^N}{N})\mathbf{1}\left\{\frac{hZ_{T}^N}{N}> l^*(T)\right\}\right)\leq
$$
\begin{equation}\label{eq:thirdtermbound}
t\int_0^{\infty}\P\left((L+h\frac{Z_{T}^N}{N})\mathbf{1}\left\{\frac{hZ_{T}^N}{N}> l^* ( T) \right\}>x\right)dx \leq 
2 t\epsilon (N, T ).
\end{equation}
To upper bound the forth term of \eqref{eq:teo3_eq2} just recall \eqref{eq:poissonboundcnt}.

Therefore, by taking $l=l^*(T)$ in \eqref{eq:teo3_eq2}, considering \eqref{eq:teo3largedev} and using the upper bounds obtained in \eqref{eq:poissonboundcnt}, \eqref{eq:firsttermbound}, \eqref{eq:secondtermbound} and \eqref{eq:thirdtermbound}, we have that the first term on the right hand side of \eqref{eq:couplingfinitelimit} is upper bounded by
\begin{equation}
\label{eq:finalbound1}
2(2L+l^*(T))C_{lip}(2L+l^*(T))\int_0^t \E\left[|U_{s-}^N(a)-U_{s-}(a)|\right]ds+
\end{equation}
$$
\epsilon (N, T )\left(2+\gamma_t\left(\frac{1}{h}+tM^<(2h)\right)+2tM^<(\gamma_t)+
\gamma_t M^<(\gamma_t) t  \right).
$$

Back to \eqref{eq:couplingfinitelimit}, note that the exchangeability of the system, assured by Assumption \ref{as:boundedandcont}, implies that
\begin{equation} \label{eq:teo3_eq4}
\frac{h}{N}\sum_{b\in \A} \int_0^t\E\left[|\phi(U_{s}^N(b))-\phi(U_{s}(b))|+|\phi(-U_{s}^N(b))-\phi(-U_{s}(b))|\right] ds=
\end{equation}
$$
h\int_0^t \E\left[\sum_{o\in \opn} |\phi(oU_{s}^N(a))-\phi(oU_{s}(a))| \right]ds.
$$
By decomposing the right-hand side of \eqref{eq:teo3_eq4} as we did on \eqref{eq:teo3largedev} and using the upper bounds obtained above for the second term on the right hand side of \eqref{eq:couplingfinitelimit}, 
we have that the right-hand side of \eqref{eq:teo3_eq4} is upper bounded by
\begin{equation}\label{eq:finalbound2}
2hC_{lip}(2L+l^*(T))\int_0^t \E\left[|U_{s}^N(a)-U_{s}(a)|\right]ds+
\end{equation}
$$
h\int_0^t \E\left[\Phi(U_{s}^N(a))\mathbf{1}\left\{\frac{hZ_{s}^N}{N}> l^*(T)\right\} \right]ds+ h\int_0^t \E\left[\Phi(U_{s}(a))\mathbf{1}\left\{\frac{hZ_{s}^N}{N}> l^*(T)\right\} \right]ds \leq
$$
$$
2hC_{lip}(2L+l^*(T))\int_0^t \E\left[|U_{s}^N(a)-U_{s}(a)|\right]ds+
\epsilon (N, T )\left(1+htM^<(2h)+htM^<(\gamma_t)\right).
$$

To finish our analysis of the term on the right-hand side of \eqref{eq:couplingfinitelimit}, we just need to study the terms on the last line of \eqref{eq:couplingfinitelimit}. First, note that Remark \ref{rem:ztabound} implies that
\begin{equation}\label{eq:ztnabound}
\frac{h}{N}\E(Z_t^N(a))\leq \frac{1}{N}(h+2htM^<(2h)).
\end{equation}
Now, recall that the martingale $M_t,$ defined in \eqref{eq:martingaledefinition}, is a sum of martingales, each associated to a compensated Poisson random measure $\tilde{\mathbf{N}}^{a,o}. $ Since these Poisson random measures are independent, the corresponding martingales are orthogonal. Therefore, using Burkholder-Davis-Gundy inequality, 
$$
\frac{h}{N}\E \left( \sup_{s \le t }|M_s| \right) \leq C \frac{h}{N}\left(\E\left(\sum_{b\in \A}\sum_{o\in \opn}\int_0^t \int_0^{\infty}\mathbf{1}\{z\leq \phi(oU_{s-}(b))\}dsdz\right)\right)^{1/2}.
$$
The exchangeability of the system together with Proposition \ref{boundlimit2} imply that the right-hand side of the equation above is upper bounded by
\begin{equation} \label{eq:mtbound}
\frac{Ch}{N^{1/2}}\left(\int_0^t\E(\Phi(U_s(a))ds\right)^{1/2} \leq \frac{Ch}{N^{1/2}}(tM^<(\gamma_t))^{1/2}.
\end{equation}
Finally, recall that $V_t$ is defined in \eqref{eq:vtdefinition}. By Jensen's inequality, 
\begin{equation}\label{eq:vtproof}
\frac{h}{N}\E \left( \sup_{s \le t }|V_s| \right) \leq \frac{h}{N}\left(\E\left(\sup_{s \le t }V_s^2\right)\right)^{1/2}.
\end{equation}
Using Jensen's inequality again, for any $s\leq t$,
\begin{multline*}
    V_s^2=
\left(\sum_{b\in \A}\sum_{o\in \opn}\int_0^s o[\E(\phi(oU_r(b)))-\phi(oU_{r}(b))]dr\right)^2
\\
\leq s\int_0^s\left(\sum_{b\in \A}\sum_{o\in \opn} o[\E(\phi(oU_r(b)))-\phi(oU_{r}(b))]\right)^2dr
\\
\leq 
t\int_0^t\left(\sum_{b\in \A}\sum_{o\in \opn} o[\E(\phi(oU_s(b)))-\phi(oU_{s}(b))]\right)^2ds.
\end{multline*}
This implies that the right-hand side of \eqref{eq:vtproof} is upper bounded by
$$
\frac{h}{N}\left[t\E\left(\int_0^t\left(\sum_{b\in \A}\sum_{o\in \opn} o[\E(\phi(oU_s(b)))-\phi(oU_{s}(b))]\right)^2ds\right)\right]^{1/2}.
$$
Since the $ ((U_s ( b))_{ s \geq 0}: b\in \A) $ are i.i.d. copies of the limit system, and recalling Proposition \ref{boundlimit1}, we conclude that the right-hand side of the last equation above is upper bounded by
\begin{equation}\label{eq:vtbound}
\frac{ht^{1/2}}{N^{1/2}}\left[\E\left(\int_0^t\left(\sum_{o\in \opn} o[\E(\phi(oU_s(a)))-\phi(oU_{s}(a))]\right)^2ds\right)\right]^{1/2} \leq \frac{ht}{N^{1/2}}(2M^<(\gamma_t)),
\end{equation}
where the last inequality above follows from the a priori bound from Proposition \ref{boundlimit1}.

Using that $ \epsilon ( N, T ) \le C_T N^{- 1/2}, $ we have 
therefore proved that for all $ t \le T,$
$$
\E\left(\sup_{s\leq t}|U_s^N(a)-U_s(a)|\right)\leq C_1(L,T,h,\phi) \int_0^t\E[|U_s^N(a)-U_s(a)|]ds+C_2(L,T,h,\phi)N^{-1/2}
$$
where $C_1$ and $C_2$ are given by putting together \eqref{eq:finalbound1}, \eqref{eq:finalbound2}, \eqref{eq:ztnabound}, \eqref{eq:mtbound} and \eqref{eq:vtbound}. 
Then Grönwall's lemma implies the assertion.
\end{proof}

\section{Proof of Theorem \ref{teo:invariant}} \label{sec:invariant}

\begin{proof}[Proof of Theorem \ref{teo:invariant}]
Consider the SDE
\begin{equation}\label{eq:auxprocess}
Y_t=Y_0-\sum_{o \in \mathcal{O}}\int_0^t \int_0^{+\infty} Y_{s-}\mathbf{1}\{z \leq \phi(o Y_{s-})\}\mathbf{N}^o(ds,dz)
+h\gamma t,
\end{equation}
where $\gamma \in \R$ is some fixed parameter. Note that, for $\gamma\neq 0,$ $(Y_t)_{t\geq 0}$ is simply a continuous time version of a house-of-cards process that jumps to $0$ at time $t>0$ with rate $\Phi(Y_t)$ and between jumps, it increases (if $\gamma>0$) or decreases (if $\gamma<0)$ linearly (a line with slope $h\gamma$). In the case $\gamma=0$, this process remains at $Y_0$ for a exponential random time with rate $\Phi(Y_0)$ and after that random time, $(Y_t)_{t\geq 0}$ gets trapped at $0$. Therefore, under Assumption \ref{as:phi3}, \eqref{eq:auxprocess} has $\delta_{0}$ as unique invariant probability measure, for $ \gamma = 0.$

Note that, for any fixed $\gamma \in \R$, \eqref{eq:auxprocess} has a path-wise unique solution for every initial condition $Y_0$ under Assumption \ref{as:phi1}. Indeed, for any $t\geq 0$
$$
|Y_t|\leq |Y_0|+h\gamma t.
$$
This implies that, conditionally on $Y_0=y_0$, $(Y_s)_{s\in[0,t]}$ can be built using an acceptance-rejection algorithm and a homogeneous Poisson process with constant rate $M^<(|y_0|+h\gamma t)<\infty$. Since this holds for any $y_0\in \R$, the pathwise uniqueness and the strong existence of \eqref{eq:auxprocess} follow.

To study the invariant probability measures of \eqref{eq:auxprocess} for $\gamma \neq 0$, let us denote
$$
\tau_x^{\gamma}=\inf\{t>0:Y_t^x=0\},
$$
where $(Y_t^x)_{t\geq 0}$ is given by \eqref{eq:auxprocess} with $Y_0^x=x$.

We start studying the case when $\gamma>0$. Note that, for any $x>0$ and for any $s>0$,
\begin{equation}\label{eq:invariantnumerator}
\P(\tau^{\gamma}_x>s)=\exp\left(-\int_0^s\Phi(x+h\gamma r)dr\right) ,
\end{equation}
and therefore, under Assumption \ref{as:phi4},
\begin{equation}\label{eq:invariantdenominator}
\E(\tau_x^{\gamma})=\frac{1}{h\gamma}\int_0^{\infty}\exp\left(-\frac{1}{h\gamma}\int_0^s\Phi(x+r)dr\right)ds<\infty.
\end{equation}
This implies that $0$ is a positive recurrent state of \eqref{eq:auxprocess} 
and therefore, \eqref{eq:auxprocess} has a unique invariant probability measure $\mu_{\gamma}$ for any $\gamma>0$.

For any measurable set $A \subset \R^+$, this unique invariant measure $\mu_{\gamma}$ satisfies
$$
\mu_{\gamma}(A)=\frac{1}{\E(\tau_0^{\gamma})}\E\left(\int_0^{\tau_0^{\gamma}}\mathbf{1}\{Y_s^0\in A\}ds \right)=\frac{1}{\E(\tau_0^{\gamma})}\E\left(\int_0^{\infty}\mathbf{1}\{h\gamma s\in A\}\mathbf{1}\{s\leq \tau_0^{\gamma}\}ds \right)=
$$
$$
\frac{1}{\E(\tau_0^{\gamma})}\frac{1}{h\gamma}\E\left(\int_0^{\infty}\mathbf{1}\{s\in A\}\mathbf{1}\left\{\frac{s}{h\gamma}\leq \tau_0^{\gamma}\right\}ds \right).
$$
Putting together \eqref{eq:invariantnumerator} and \eqref{eq:invariantdenominator}, we conclude that $\mu_{\gamma}$ possesses a Lebesgue density $ g_{\gamma}(x)$ given by
\begin{equation} \label{eq:invariantmeasure}
g_{\gamma}(x)= \exp\left(-\frac{1}{h\gamma}\int_0^x\Phi(r)dr\right)\Big/ \int_0^{\infty}\exp\left(-\frac{1}{h\gamma}\int_0^s\Phi(r)dr\right)ds.
\end{equation}
Using the symmetry of the function $\Phi$ we conclude that for any $x\leq 0$ and $\gamma<0$,
$$
g_{\gamma}(x)=g_{-\gamma}(-x).
$$
is the density of the unique invariant probability measure of \eqref{eq:auxprocess} in the case $\gamma <0$.

Now, let us go back to the original process \eqref{eq:limiteq}. Consider an invariant probability measure $g(dx)$, supported in $\R$ for the SDE solution to \eqref{eq:limiteq}. If $U_0 \sim g$,  then $\E(\phi(U_t)-\phi(-U_t))=\int_x (\phi(x)-\phi(-x)) g(dx)$, for any $t>0$.
Therefore, if $g_{\gamma}$ is an invariant probability measure for the process \eqref{eq:auxprocess}, for $\gamma \in \R$, and $\gamma$ satisfies 
\begin{equation}
\label{gamma}
\gamma=\int_x (\phi(x)-\phi(-x)) g_{\gamma}(dx),
\end{equation}
then $g_{\gamma}$ is a invariant probability measure for the SDE \eqref{eq:limiteq}. This concludes the proof of Part 3 of Theorem \ref{teo:invariant}. 


If $\gamma=0$, under Assumption \ref{as:phi3}, $\delta_0$ is the unique invariant probability measure for \eqref{eq:auxprocess}. Since $\delta_0$ satisfies \eqref{gamma} for $\gamma=0$, we conclude that Part 1 of Theorem \ref{teo:invariant} holds. By noticing  that $g$ must have support in $[0,\infty)$, if $\gamma>0$ and in $(-\infty,0]$, if $\gamma>0$, we conclude the proof of Part 2 of Theorem \ref{teo:invariant}.

Moreover, note that if $g_{\gamma}$ satisfies \eqref{gamma}, then 
$$
-\gamma=-\int_0^{\infty} (\phi(x)-\phi(-x))g_{\gamma}(dx)=\int_{-\infty}^0 (\phi(x)-\phi(-x))g_{-\gamma}(dx).
$$
With this we conclude the proof of Part 4 of Theorem \ref{teo:invariant}.

To prove Part 5 of Theorem \ref{teo:invariant}, note that for $\gamma>0$, by  putting together \eqref{eq:invariantmeasure} and \eqref{gamma}, we have that $g_{\gamma}$ is an invariant probability measure of \eqref{eq:limiteq} if
\begin{equation} \label{eq:invariantmeasureeq}
\int_0^{\infty} (\phi(x)-\phi(-x))\exp\left(-\frac{1}{h\gamma}\int_0^x\Phi(r)dr\right)dx=\gamma \int_0^{\infty}\exp\left(-\frac{1}{h\gamma}\int_0^x\Phi(r)dr\right)dx.
\end{equation}

By considering  Assumption \ref{as:sym}, note that $\phi(x)+\phi(-x)=2B$ and $\phi(x)-\phi(-x)=2f(x)$. Therefore, in this case, \eqref{eq:invariantmeasureeq} reads as 
$$
\int_0^{\infty}\exp\left(-(h\gamma)^{-1}2Bx\right)2f(x)dx=
\gamma\int_0^{\infty}\exp\left(-(h\gamma)^{-1}2Bx\right)dx
$$
which is equivalent to 
$$
\int_0^{\infty}\exp\left(-(h\gamma)^{-1}2Bx\right)f(x)dx=
\frac{\gamma^2h}{4B}.
$$
Making a change of variables, we obtain
\begin{equation}\label{eq:oddfixedpointeq}
\int_0^{\infty}\exp\left(-u\right)f\left(\frac{uh\gamma}{2B}\right)du=
\frac{\gamma}{2}.   
\end{equation}
Observing that  $f':\R \to \R^+$ is bounded, by the Dominated Convergence Theorem, the derivative of the left-hand side of \eqref{eq:oddfixedpointeq} with respect to $\gamma$ is given by
\begin{equation}\label{eq:oddderivative}
\frac{h}{2B}\int_0^{\infty}\exp\left(-u\right)uf'\left(\frac{uh\gamma}{2B}\right)du.
\end{equation}
Evaluating \eqref{eq:oddderivative} when $\gamma=0$, we obtain
$$
\frac{h}{2B}\int_0^{\infty}\exp\left(-u\right)uf'(0)du=\frac{hf'(0)}{2B}.
$$
Note that $f'(x)\to 0$ as $x\to \infty$ since $f$ is bounded. By the Dominated Convergence Theorem,
$$
\frac{h}{2B}\int_0^{\infty}\exp\left(-u\right)uf'\left(\frac{uh\gamma}{2B}\right)du \to 0, \text{ as } \gamma \to \infty.
$$
This implies that if $h>B/f'(0)$, \eqref{eq:oddfixedpointeq} has at least one positive fixed point. 
Moreover,
since $f'(x)$ is  a strictly decreasing function in $[0,\infty)$, we have that \eqref{eq:oddderivative}
is a strictly decreasing function of $\gamma$. We conclude that if $h>L/f'(0)$, \eqref{eq:oddfixedpointeq} has a unique fixed point $\gamma^*_h>0$ and if $h \leq  L/f'(0)$, \eqref{eq:oddfixedpointeq} does not have a positive fixed point. In the case $h>L/f'(0)$, the symmetry properties of the process imply that $\gamma^*_h$, $-\gamma^*_h$ and $0$ are the unique solutions to \eqref{gamma}. This concludes the proof of Part 5 of Theorem \ref{teo:invariant}.
    
\end{proof}

\section*{Acknowledgments}

This work was produced as part of the activities of FAPESP  Research, Innovation and Dissemination Center for Neuromathematics (grant \# 2013/ 07699-0 , S.Paulo Research Foundation (FAPESP).
KL was successively supported by FAPESP fellowship (grant 2022/07386-0 and 2023/12335-9). The authors acknowledge support of the Institut Henri Poincar\'e (UAR 839 CNRS-Sorbonne Universit\'e), and LabEx CARMIN (ANR-10-LABX-59-01). We thank Antonio Galves for interesting discussions in the beginning of this project.

\bibliographystyle{apalike}
\bibliography{bib}

\newpage

\noindent
Eva Löcherbach \\ 
Statistique, Analyse et Mod\'elisation Multidisciplinaire \\ Universit\'e Paris 1 Panth\'eon-Sorbonne \\ 
EA 4543 et FR FP2M 2036 CNRS \\ France \\ e-mail address: eva.locherbach@univ-paris1.fr

\vspace{0.5cm}

\noindent
Kádmo de Souza Laxa \\ Faculdade de Filosofia, Ciências e Letras de Ribeirão Preto  \\ Universidade de São Paulo \\
Av. Bandeirantes, 3900\\
Ribeirão Preto-SP, 14040-901 \\
Brazil \\
e-mail address: \texttt{kadmo.laxa@usp.br}

\end{document}